\documentclass[review]{elsarticle}
\usepackage{amsthm,amsmath,amsfonts,amssymb}
\usepackage{lineno, hyperref}
\usepackage{color}
\usepackage[normalem]{ulem}
\usepackage{mathrsfs}
\usepackage{booktabs}
\usepackage{titlesec}
\usepackage{geometry}
\usepackage{subcaption}
\modulolinenumbers[5]
\numberwithin{equation}{section} 

\usepackage{tikz}
\usetikzlibrary{arrows.meta,positioning,shapes.geometric,fit,calc}
\usepackage{pgfplots}
\usetikzlibrary{arrows}
\usetikzlibrary{arrows.meta}
\usepackage{algorithm}
\usepackage{algorithmic}
\pgfplotsset{compat=1.18}
\titleformat{\section}[hang]{\normalfont\Large\bfseries}{\thesection\ \textbar}{1em}{\MakeUppercase}
\newtheorem{Theorem}{Theorem}[section]
\newtheorem{Definition}[Theorem]{Definition}
\newtheorem{Proposition}[Theorem]{Proposition}
\newtheorem{Lemma}[Theorem]{Lemma}
\newtheorem{corollary}[Theorem]{Corollary}
\newtheorem{remark}[Theorem]{Remark}

\newtheorem{assumption}[Theorem]{Assumption}

\newcommand{\vertiii}[1]{{\left\vert\kern-0.25ex\left\vert\kern-0.25ex\left\vert #1
		\right\vert\kern-0.25ex\right\vert\kern-0.25ex\right\vert}}
\newcommand{\be}{\begin{equation}\begin{aligned}}
		\newcommand{\ee}{\end{aligned}\end{equation}}
\newcommand{\ben}{\begin{equation}\nonumber\begin{aligned}}

		\usepackage{hyperref}
		\usepackage{graphicx}
		\usepackage{placeins}
		\usepackage[normalem]{ulem}
		\usepackage{color}
		\usepackage{cancel}
		\usepackage{mathptmx}      
		\usepackage{mathtools}
		\usepackage{tikz}
		
		\def\e1{{\varepsilon_{11}}}
		\def\b1{{\beta_{11}}}
		\def\bp3{{\beta_{33}}}
		\def\ep3{{\varepsilon_{33}}}

		\journal{xxxx}
\begin{document}
			
			\begin{frontmatter}
\title{{\bf Boundary Control of a Modified Timoshenko–Ehrenfest Beam via Reduced-Order Reinforcement Learning  }}

\author{Leonardo Rogério da Silva Rodrigues}
\address{CCCo, Science Center, Federal University of Maranhão, Dr. Anselmo Street, 2008, 65400-00, Maranhão, Brazil\\
leonardo.rodrigues@ufma.br}

\begin{abstract}
 
In this paper, we study the active vibration control of an elastic beam using reinforcement learning (RL). The physical system is modeled by a modified Timoshenko-Ehrenfest type linear partial differential equation with simply supported boundary conditions. The dynamics are approximated by a reduced-order model (ROM) obtained through eigenvalue expansion and truncated Galerkin projection. Control is applied through boundary actuation.
The resulting finite-dimensional system is formulated as a continuous-time Markov decision process and solved numerically using the Proximal Policy Optimization (PPO) algorithm. All modeling steps, from physical formulation to temporal discretization and control design, are explicitly described.
A systematic comparison is performed between the controller learned via PPO and the classical linear quadratic (LQR) regulator. The analysis highlights the strengths and limitations of each approach, both from a theoretical and practical point of view, with particular emphasis on energy dissipation mechanisms and actuator constraints.
\end{abstract}

 \end{frontmatter}
\newpage

\section{Introduction}\label{Section1}

Control of dynamical systems governed by partial differential equations (PDEs) is a challenging task, both from a theoretical and a computational perspective, and has numerous important applications, such as turbulence control \cite{Bewley2001}, chemical process control \cite{Christofides2001}, engineering systems \cite{daSilvaRodrigues2025}, and robotic systems \cite{Zhang2005}. The application of machine learning methods to systems modeled by partial differential equations is computationally demanding due to the high dimensionality inherent in the dynamics, both in terms of the system state and the control variables. This is because it involves the manipulation and learning of a large number of degrees of freedom \cite{Brunton2020}.

In recent years, several data-driven and machine learning–based approaches have been proposed for the control of systems governed by partial differential equations (PDEs), particularly in the field of fluid mechanics \cite{Duriez2017}. Among these methodologies, reinforcement learning (RL) has played an increasingly important role \cite{Konishi2022, Novati2021,Varela2022,Vinuesa2022,Wang2022,Werner2023}. However, the complexity of the dynamics and the extremely large number of parameters require complementary strategies, such as the investigation of patterns intrinsic to the dynamics or the explicit incorporation of system knowledge.

According to \cite{SuttonBarto2018}, the standard reinforcement learning (RL) formulation consists of an agent interacting with an environment (stochastic or deterministic) over discrete time steps. At each time step $k$, the agent receives an observation $(s_k\in\mathcal{S})$, executes an action $(a_k\in\mathcal{A})$, and receives a reward $(r_k\in\mathbb{R})$. The environment may be only partially observable, and observations may also include time-delayed coordinates.

The agent's behavior is defined by a policy $\pi$, namely a mapping from the state space to a probability distribution over actions,
\[
\pi : \mathcal{S} \longrightarrow \mathcal{P}(\mathcal{A}).
\]

The environment dynamics are described by a Markov decision process (MDP), characterized by a state space $\mathcal{S}$, an action space $\mathcal{A}\subseteq\mathbb{R}^m$, an initial-state distribution $P(s_0)$, a transition model $P(s_{k+1}\mid s_k,a_k)$, and a reward function $r(s_k,a_k)$.

Given a policy $\pi$, the return from a state is defined as the expected sum of discounted future rewards,
\[
\mathbb{G} = \mathbb{E}_{\pi}\!\left[
\sum_{i=k}^{\infty} \gamma^{\,i-k} r_i
\right],
\]
where $\gamma\in[0,1]$ is the discount factor. The goal in RL is to learn a policy that maximizes the expected return, which is closely related to Bellman's principle of optimality and the associated value function.

The use of RL in systems governed by partial differential equations (PDEs) has attracted increasing attention as an alternative or complement to classical optimal control and model reduction methods. The literature includes several contributions along the following lines: (i) stabilization of reduced-order models via RL, (ii) RL as a direct physical controller in mechanical systems, (iii) RL applications in high-dimensional computational mechanics, and (iv) learning-based methods for approximating value equations associated with optimal control.

One of the early rigorous framings of RL in the PDE context appears in \cite{BENOSMAN-CHAKRABARTY-BORGGAARD}, which proposes RL for stabilizing reduced-order models (ROMs) obtained via eigenmode expansions or Proper Orthogonal Decomposition (POD)--Galerkin reductions. In that work, RL acts as a closure mechanism by learning artificial terms that compensate for instability induced by modal truncation. The problem is formulated as a discrete-time dynamical system, and the closure term is interpreted as a virtual controller learned via adaptive dynamic programming (ADP), with practical stability guarantees under Lipschitz assumptions and linear matrix inequality (LMI) conditions.

A complementary line of work uses RL as a direct physical controller, replacing or augmenting classical control laws. A representative example is \cite{HWANGBO-QUADROTOR}, where neural policies are trained to control a quadrotor with continuous actuation, achieving practical stability even from challenging initial conditions. In this setting, RL directly drives the system actuators, learning state and action feedback policies without explicitly relying on linear models or Riccati equations.
The study provides empirical evidence that RL-based controllers may exhibit robustness to disturbances and parametric uncertainty compared with classical linear controllers, such as the linear quadratic regulator (LQR) or linear model predictive control (MPC), although typically without global analytical stability guarantees.

A panoramic view of the state of the art is provided by the review on deep reinforcement learning in fluid mechanics \cite{Rabault2021}, which systematizes DRL applications in control, optimization, and modeling of PDE governed systems. That work emphasizes practical aspects that are crucial for RL success in physical environments, such as reward shaping based on physically meaningful quantities (energy, drag, vibration), the use of partial observations via sensors, and training acceleration through multiple parallel environments.
The review highlights proximal policy optimization (PPO) as one of the most robust and widely used approaches for continuous control problems, due to its numerical stability and implementation simplicity. These findings support the suitability of modern RL for control problems in computational mechanics, including beyond the specific fluid mechanics setting.

A conceptually different direction is presented in ``Neural Q-learning for solving PDEs'' \cite{NEURAL-Q-PDE}, where the authors propose a mesh free approach for solving elliptic PDEs inspired by Q-learning. In that work, neural networks are trained to approximate the PDE solution directly, using biased gradient updates analogous to classical Q-learning, thereby avoiding explicit differentiation of the differential operator.
From an optimal control perspective, this is particularly relevant because the Hamilton--Jacobi--Bellman (HJB) equations associated with continuous optimal control are high dimensional PDEs, for which grid based methods are typically infeasible. The paper shows, under monotonicity assumptions and in the wide network regime (i.e., in the neural tangent kernel (NTK) limit), that training converges to the PDE solution in the $L^2$ norm. Thus, it provides a solid theoretical foundation for learning based methods applied directly to value equations, complementing policy gradient approaches such as PPO.

Recent approaches (see \cite{PEITZ2024134096}) propose convolutional reinforcement learning for PDE control, exploiting spatial locality and translation invariance. In this perspective, the PDE is spatially discretized and reinterpreted as a set of coupled local subsystems, each controlled by an identical agent that observes only a restricted spatial neighborhood of the state, yielding a distributed control viewpoint. Complexity is reduced through local convolutional observations. Moreover, convolutional learning based methods can achieve excellent spatial scalability and are particularly well suited to dissipative systems with local information transport.

The present work adopts a complementary philosophy. Unlike artificial closure approaches for ROM stabilization, we use reinforcement learning as a physical controller applied at the boundary of a structural PDE system, rather than as distributed control over the spatial domain. Here, complexity reduction is not achieved through local convolutional observations, but through spectral projection onto a low dimensional modal subspace, preserving the energy structure and physical interpretability of the system in contrast to purely empirical methods, and explicitly linking the reward to modal energy dissipation. The modal approach is especially appropriate for structural systems in which a few global modes dominate the dynamics, as in flexible beams and Timoshenko-type models.
Finally, we systematically compare the controller learned via PPO with the linear quadratic regulator (LQR), highlighting the advantages and limitations of each paradigm in light of both classical theory and recent empirical evidence on reinforcement learning based control in mechanics.

This paper designs a nonlinear learned boundary feedback law for a reduced-order model of a modified Timoshenko–Ehrenfest beam \cite{ALMEIDA-RAMOS,ALMEIDA-RAMOS-SANTOS-M,ELISHAKOFF,ELISHAKOFF-1}.
The analysis is grounded on the following partial differential equation:
 \begin{eqnarray}\label{eq:MTEB}
	\rho_{1}\varphi_{tt}-\kappa\left(\varphi_{x}+\psi\right)_{x}&=&0, \\
	-\rho_{2}\varphi_{ttx}-b\psi_{xx}+\kappa\left(\varphi_{x}+\psi\right)&=&0, \label{eq5.int}
\end{eqnarray}
The modified formulation provides a practical and reliable foundation for the present study, as it eliminates the so-called second spectrum anomaly, while preserving the essential mechanical features of the beam dynamics.

Moreover, the resulting decoupled structure is particularly advantageous in the context of control, as it facilitates subsequent reduced-order modeling (ROM) procedures and leads to a clearer formulation of optimality conditions. These properties make the modified formulation especially well suited for data-driven and learning-based control approaches.
Under the structural assumptions described above, the system admits the following reduced scalar representation for the transverse displacement:
\begin{equation}\label{VTT}
	\varphi_{tt}
	-
	\left(\frac{b}{\kappa}+\frac{\rho_{2}}{\rho_{1}}\right)\varphi_{ttxx}
	+
	\frac{b}{\rho_{1}} \varphi_{xxxx}
	=
	0,
\end{equation}

The central objective of this work is to introduce a physically meaningful boundary controller acting on system~\eqref{VTT}, with the aim of establishing a direct connection between the reward function used in learning-based control and the dissipation of mechanical energy in the reduced-order model. Within this framework, reinforcement learning is adopted as a nonlinear feedback synthesis mechanism, and the agent is trained using the Proximal Policy Optimization (PPO) algorithm \cite{Schulman2017PPO}.

PPO is a model-free, actor–critic reinforcement learning algorithm that learns a stochastic policy in a stable manner through proximal updates, thereby preventing excessively large policy changes at each training step. This property is particularly important in control problems governed by partial differential equations, where stability and robustness of the learned feedback law are critical. By designing the reward function to explicitly penalize the modal energy of the reduced system and the control effort, the learning process is naturally guided toward policies that promote vibration attenuation through energy dissipation.

Within this setting, we propose a systematic framework for applying reinforcement learning techniques to boundary-controlled beam models, highlighting the interplay between physical modeling, reduced-order representations, and modern machine learning-based control methodologies.

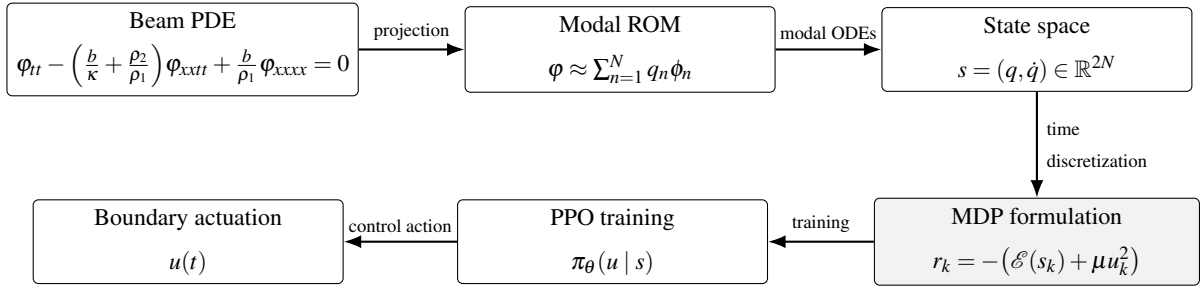
\begin{figure}[ht]
	\centering
	\begin{tikzpicture}[
		>=Latex,
		box/.style={
			draw,
			rounded corners=2pt,
			align=center,
			minimum width=4.1cm,
			minimum height=1.05cm,
			font=\small
		},
		sbox/.style={
			draw,
			rounded corners=2pt,
			align=center,
			minimum width=4.3cm,
			minimum height=1.15cm,
			fill=gray!10,
			font=\small
		},
		arrow/.style={->, thick}
		]
		
		\node[box] (pde) {
			Beam PDE\\
			$\varphi_{tt}
			-\Big(\frac{b}{\kappa}+\frac{\rho_2}{\rho_1}\Big)\varphi_{xxtt}
			+\frac{b}{\rho_{1}}\varphi_{xxxx}=0$
		};
		
		\node[box, right=14mm of pde] (rom) {
			Modal ROM\\
			$\varphi\approx\sum_{n=1}^N q_n\phi_n$
		};
		
		\node[box, right=14mm of rom] (ss) {
			State space\\
			$s=(q,\dot q)\in\mathbb{R}^{2N}$
		};
		
		\node[sbox, below=14mm of ss] (mdp) {
			MDP formulation\\
			$r_k=-\big(\mathcal{E}(s_k)+\mu u_k^2\big)$
		};
		
		\node[box, left=14mm of mdp] (ppo) {
			PPO training\\
			$\pi_\theta(u\mid s)$
		};
		
		\node[box, left=15mm of ppo] (bc) {
			Boundary actuation\\
			$u(t)$
		};
		
		\draw[arrow] (pde) -- node[above, font=\scriptsize, align=center]
		{projection} (rom);
		
		\draw[arrow] (rom) -- node[above, font=\scriptsize, align=center]
		{modal ODEs} (ss);
		
		\draw[arrow] (ss) -- node[right, font=\scriptsize, align=left]
		{time\\discretization} (mdp);
		
		\draw[arrow] (mdp) -- node[above, font=\scriptsize]
		{training} (ppo);
		
		\draw[arrow] (ppo) -- node[above, font=\scriptsize]
		{control action} (bc);
		
	\end{tikzpicture}
	
	\caption{
Pipeline of the proposed approach in a two-row layout.
The beam PDE is projected onto a truncated modal subspace, yielding a finite-dimensional reduced-order model (ROM).
The dynamics are discretized and formulated as a Markov decision process (MDP) with an energy-based reward.
A PPO agent learns a boundary control policy acting at the simply supported beam endpoints.
	}
	\label{fig:pipeline_zigzag}
\end{figure}
\newpage
\subsection{Plan of the paper}

The structure of this paper is as follows. Section 2 presents the continuous beam model and its energy properties, including the boundary actuation formulation and the lifting transformation. 
Section 3 derives a reduced-order modal model via Galerkin truncation,  formulates the corresponding optimal control problem, and recasts it as a Markov decision process. Section 4 presents the reinforcement learning algorithm based on Proximal Policy Optimization (PPO). 
Section 5 discusses the energy analysis and stability properties of the 
closed-loop system. Section 6 provides numerical results and comparisons with an LQR baseline. 
Finally, Section 7 concludes the paper and discusses the main findings.

\section{Continuous beam model}

This section introduces the continuous model of the \emph{modified} Timoshenko--Ehrenfest beam and the boundary actuation mechanism considered in this work. 
Although the modified formulation can be written as a coupled system in the transverse displacement $\varphi$ and the rotation field $\psi$ \cite{ALMEIDA-RAMOS,ALMEIDA-RAMOS-SANTOS-M,ELISHAKOFF,ELISHAKOFF-1}, our analysis and control design will be carried out on an equivalent decoupled scalar equation, which is more convenient for model reduction and learning-based feedback synthesis.

\subsection*{Modified Timoshenko--Ehrenfest formulation and decoupling}
Let $L>0$ and consider a homogeneous, linearly elastic beam occupying $(0,L)$. 
We denote by $\varphi(x,t)$ the transverse displacement and by $\psi(x,t)$ the rotation variable. 
In the absence of distributed forces, the modified Timoshenko--Ehrenfest dynamics read
\begin{equation}
	\label{eq:MTEB_sys}
	\begin{cases}
		\rho_{1}\,\varphi_{tt}-\kappa(\varphi_{x}+\psi)_{x}=0,\\[1mm]
		-\rho_{2}\,\varphi_{ttx}-b\,\psi_{xx}+\kappa(\varphi_{x}+\psi)=0,
	\end{cases}
	\qquad (x,t)\in(0,L)\times(0,T),
\end{equation}
where $\rho_{1}=\rho A$ and $\rho_{2}=\rho I$ are, respectively, the translational and rotational inertia densities, $\kappa\in(0,1]$ is the shear correction factor, and $EI$ is the bending stiffness.

The modified formulation is adopted because it eliminates the so-called \emph{second spectrum anomaly} while preserving the main mechanical features of shear and rotary inertia effects. Moreover, it admits a decoupled scalar representation for $\varphi$, which is particularly advantageous for reduced-order modeling and boundary control.

Under the structural assumptions above, the transverse displacement satisfies
\begin{equation}
	\label{eq:beam}
	\varphi_{tt}
	-
	\alpha\,\varphi_{ttxx}
	+
	\beta\,\varphi_{xxxx}
	=
	0,
	\qquad (x,t)\in(0,L)\times(0,T),
\end{equation}
with
\begin{equation}
	\label{eq:alphabeta}
	\alpha=\frac{b}{\kappa}+\frac{\rho_2}{\rho_1},
	\qquad
	\beta=\frac{b}{\rho_1}.
\end{equation}
In what follows, we work exclusively with \eqref{eq:beam}, and we refer to it as the \emph{decoupled modified beam equation}.

\subsection*{Initial conditions}
The initial state is prescribed as
\begin{equation}
	\label{eq:IC}
	\varphi(x,0)=\varphi_0(x),
	\qquad
	\varphi_t(x,0)=\varphi_1(x),
\end{equation}
where $\varphi_0$ and $\varphi_1$ are sufficiently regular functions.

\subsection*{Energy of the continuous system}

\begin{Definition}[Total energy]
	\label{def:energy_cont}
	The total energy associated with a regular solution of \eqref{eq:beam} is
	\begin{equation}
		\label{eq:energy_cont}
		\mathcal{E}(t)
		=
		\frac{1}{2}
		\int_0^L
		\Big(
		|\varphi_t|^2
		+
		\alpha\,|\varphi_{xt}|^2
		+
		\beta\,|\varphi_{xx}|^2
		\Big)\,dx.
	\end{equation}
\end{Definition}

When the boundary input is homogeneous ($u\equiv 0$), the energy \eqref{eq:energy_cont} is conserved, reflecting the conservative nature of the continuous dynamics.

\subsection*{Boundary conditions and boundary control}

We assume that the beam is simply supported and actuated through the boundary curvature:
\begin{align}
	\varphi(0,t) &= \varphi(L,t) = 0, \label{eq:BC1}\\
	\varphi_{xx}(0,t) &= \varphi_{xx}(L,t) = u(t), \label{eq:BC2}
\end{align}
where $u:[0,T]\to\mathbb{R}$ is the control signal.

\subsection*{Well-posedness (homogeneous boundary conditions)}
For the analytical foundation of the reduced-order construction, we recall that the homogeneous simply supported problem (i.e., $u\equiv 0$ in \eqref{eq:BC2}) is well posed and admits an energy identity of the form \eqref{eq:energy_cont}. 
An equivalent formulation can also be stated at the level of the coupled system \eqref{eq:MTEB_sys}, together with a reconstruction of $\psi$ via an elliptic Neumann problem; see Theorem~\ref{thm:wellposed_mteb_ss}.

\begin{Theorem}[Existence and uniqueness]\label{thm:wellposed_mteb_ss}
	Let $L>0$ and physical constants
	\[
	\rho_{1}>0,\qquad \rho_{2}>0,\qquad \kappa>0,\qquad b>0.
	\]
	Consider the modified Timoshenko--Ehrenfest system
	\begin{equation}
		\label{eq:MTEB_sys_neu}
		\begin{cases}
			\rho_{1}\,\varphi_{tt}-\kappa(\varphi_{x}+\psi)_{x}=0,\\[1mm]
			-\rho_{2}\,\varphi_{ttx}-b\,\psi_{xx}+\kappa(\varphi_{x}+\psi)=0,
		\end{cases}
		\qquad (x,t)\in(0,L)\times(0,T),
	\end{equation}
	with boundary conditions 
	\begin{equation}
		\label{eq:bc_phi_psi}
		\varphi(0,t)=\varphi(L,t)=0,\qquad 
		\psi_{x}(0,t)=\psi_{x}(L,t)=0,
		\qquad t\in[0,T],
	\end{equation}
	and initial data
	\begin{equation}
		\label{eq:ic_phi}
		\varphi(\cdot,0)=\varphi_{0},\qquad \varphi_{t}(\cdot,0)=\varphi_{1}.
	\end{equation}
	
	Define
	\[
	\alpha:=\frac{b}{\kappa}+\frac{\rho_{2}}{\rho_{1}}>0,
	\qquad
	\beta:=\frac{b}{\rho_{1}}>0.
	\]
	
	Assume
	\[
	\varphi_{0}\in D(A):=\Big\{u\in H^{4}(0,L)\cap H^{1}_{0}(0,L)\;:\; u_{xx}(0)=u_{xx}(L)=0\Big\},
	\qquad 
	\varphi_{1}\in V:=H^{2}(0,L)\cap H^{1}_{0}(0,L).
	\]
	
	Then:
	
	\begin{enumerate}
		\item[(i)] There exists a unique solution
		\[
		\varphi \in C\!\big([0,T];D(A)\big)\cap C^{1}\!\big([0,T];V\big)\cap C^{2}\!\big([0,T];H^{1}_{0}(0,L)\big)
		\]
		to the decoupled scalar equation
		\begin{equation}
			\label{eq:VTT_ss}
			\varphi_{tt}-\alpha\,\varphi_{ttxx}+\beta\,\varphi_{xxxx}=0,
			\qquad (x,t)\in(0,L)\times(0,T),
		\end{equation}
		subject to the simply supported boundary conditions
		\begin{align}
			\varphi(0,t)&=\varphi(L,t)=0, \label{eq:BC1_ss}\\
			\varphi_{xx}(0,t)&=\varphi_{xx}(L,t)=0, \label{eq:BC2_ss}
		\end{align}
		and the initial conditions \eqref{eq:ic_phi}.
		
		\item[(ii)] For each $t\in[0,T]$, define $\psi(\cdot,t)$ as the unique solution of the elliptic (Neumann) problem
		\begin{equation}
			\label{eq:psi_recovery_neu}
			-b\,\psi_{xx}+\kappa\,\psi
			=
			\rho_{2}\,\varphi_{ttx}-\kappa\,\varphi_{x}
			\quad\text{in }(0,L),
			\qquad \psi_{x}(0,t)=\psi_{x}(L,t)=0.
		\end{equation}
		Then
		\[
		\psi \in C\!\big([0,T];H^{1}(0,L)\big),
		\]
		and the pair $(\varphi,\psi)$ satisfies system \eqref{eq:MTEB_sys_neu} in the weak sense.
		
		\item[(iii)] (Energy conservation and continuous dependence) The energy
		\begin{equation}
			\label{eq:energy_ss}
			\mathcal{E}(t):=\frac12\Big(
			\|\varphi_{t}(t)\|^{2}_{L^{2}}
			+\alpha\,\|\varphi_{tx}(t)\|^{2}_{L^{2}}
			+\beta\,\|\varphi_{xx}(t)\|^{2}_{L^{2}}
			\Big)
		\end{equation}
		is constant for $t\in[0,T]$.
		In particular, there exists $C=C(\alpha,\beta,L)$ such that
		\[
		\sup_{t\in[0,T]}\Big(\|\varphi_{t}(t)\|_{H^{1}}+\|\varphi(t)\|_{H^{2}}\Big)
		\le
		C\Big(\|\varphi_{1}\|_{H^{1}}+\|\varphi_{0}\|_{H^{2}}\Big),
		\]
		and therefore the problem is well posed in the sense of Hadamard.
	\end{enumerate}
\end{Theorem}

\begin{proof}
	\textbf{Step 1: Operator formulation and domain of the biharmonic operator.}
	Rewrite \eqref{eq:VTT_ss} as
	\[
	M\varphi_{tt}+A\varphi=0,
	\qquad
	M:=I-\alpha\,\partial_{xx},
	\qquad
	A:=\beta\,\partial_{xxxx}.
	\]
	Under boundary conditions \eqref{eq:BC1_ss}--\eqref{eq:BC2_ss}, the operator $A$ is self-adjoint and nonnegative in $L^{2}(0,L)$ with domain $D(A)$ defined above.
	Moreover, $M$ is a coercive isomorphism from $H^{1}_{0}(0,L)$ onto $H^{-1}(0,L)$.
	
	\textbf{Step 2: Existence and uniqueness for the decoupled equation.}
	Applying the Galerkin method using a spectral basis of $A$ (simply supported modes), or equivalently the standard semigroup theory for second-order evolution equations of the form $M\ddot\varphi +A\varphi=0$, yields existence and uniqueness of $\varphi$ with the regularity stated in (i).
	
	\textbf{Step 3: Energy identity.}
	Multiplying \eqref{eq:VTT_ss} by $\varphi_{t}$ and integrating over $(0,L)$ gives
	\[
	\frac{d}{dt}\Big(\tfrac12\|\varphi_{t}\|^{2}_{L^{2}}+\tfrac{\alpha}{2}\|\varphi_{tx}\|^{2}_{L^{2}}\Big)
	+
	\beta\int_{0}^{L}\varphi_{xxxx}\varphi_{t}\,dx
	=0.
	\]
	Integrating by parts twice,
	\[
	\int_{0}^{L}\varphi_{xxxx}\varphi_{t}\,dx
	=
	\frac{d}{dt}\Big(\tfrac12\|\varphi_{xx}\|^{2}_{L^{2}}\Big)
	+\Big[\varphi_{xxx}\varphi_{t}-\varphi_{xx}\varphi_{tx}\Big]_{0}^{L}.
	\]
	Since $\varphi(0,t)=\varphi(L,t)=0$ implies $\varphi_{t}(0,t)=\varphi_{t}(L,t)=0$,
	and $\varphi_{xx}(0,t)=\varphi_{xx}(L,t)=0$,
	the boundary term vanishes.
	Hence $\frac{d}{dt}\mathcal{E}(t)=0$, proving \eqref{eq:energy_ss}.
	
	\textbf{Step 4: Reconstruction of $\psi$ (Neumann case).}
	For each $t$, the right-hand side of \eqref{eq:psi_recovery_neu} belongs to $H^{-1}(0,L)$
	since $\varphi_{x}\in L^{2}$ and $\varphi_{ttx}\in H^{-1}$.
	The bilinear form
	\[
	a(\psi,\eta):=b\int_{0}^{L}\psi_{x}\eta_{x}\,dx+\kappa\int_{0}^{L}\psi\eta\,dx,
	\qquad \psi,\eta\in H^{1}(0,L),
	\]
	is continuous and coercive in $H^{1}(0,L)$ (because $\kappa>0$).
	Therefore, by the Lax--Milgram theorem \cite{EVANS2010}, there exists a unique $\psi(\cdot,t)\in H^{1}(0,L)$
	solving \eqref{eq:psi_recovery_neu} with Neumann boundary conditions.
	Time continuity follows from the continuity of the right-hand side.
	By construction, $(\varphi,\psi)$ satisfies \eqref{eq:MTEB_sys_neu} in the weak sense.
\end{proof}

\subsection*{Lifting of boundary control}
\label{subsec:lifting}

Since the control enters through the non-homogeneous boundary condition \eqref{eq:BC2}, the spatial operator associated with \eqref{eq:beam} does not have a fixed domain when $u$ varies in time. 
As a consequence, a direct modal decomposition on a time-independent eigenbasis is not available.

To overcome this difficulty, we introduce a lifting transformation of the form
\[
\varphi(x,t)=w(x,t)-u(t)\,g(x),
\qquad g(x)=\frac{x(L-x)}{2}.
\]
This choice ensures that $g(0)=g(L)=0$, $g_{xx}=-1$ and $g_{xxxx}=0$, so that the new variable
$w$ satisfies homogeneous boundary conditions.
As a result, the boundary control is converted into an equivalent source term
in the interior of the domain, and the governing equation for $w$ becomes a
homogeneous boundary-value problem with a fixed domain:
\[
w(0,t)=w(L,t)=0,
\qquad
w_{xx}(0,t)=w_{xx}(L,t)=0.
\]

Substituting this decomposition into \eqref{eq:beam} yields the forced homogeneous
problem
\begin{equation}\label{eq:lifted_w}
	w_{tt}-\alpha w_{ttxx}+\beta w_{xxxx}
	=
	(g(x)+\alpha)\,\ddot u(t),
\end{equation}
where the control effect appears explicitly as a source term depending on the
second time derivative of $u$.

\subsection*{Modal expansion and truncation}

Let $\{\phi_n\}_{n\ge1}$ be the orthonormal basis of $L^2(0,L)$ defined by
\[
\phi_n(x)=\sqrt{\frac{2}{L}}\sin\!\left(\frac{n\pi x}{L}\right),
\qquad
\lambda_n=\left(\frac{n\pi}{L}\right)^4,
\]
which satisfies
\[
\phi_n(0)=\phi_n(L)=0,
\qquad
\phi_n''(0)=\phi_n''(L)=0.
\]

We approximate the solution $w$ by an $N$-mode truncation:
\[
w_N(x,t)=\sum_{n=1}^N q_n(t)\,\phi_n(x).
\]

\subsection*{Reduced modal equations}

Substituting the modal expansion into \eqref{eq:lifted_w}, multiplying by $\phi_n$
and integrating over $(0,L)$, we obtain, for each mode,
\begin{equation}
	(1+\alpha k_n^2)\,\ddot q_n(t)
	+
	\beta k_n^4\,q_n(t)
	=
	\gamma_n\,\ddot u(t),
	\qquad
	k_n=\frac{n\pi}{L},
\end{equation}
where
\[
\gamma_n=\frac{2}{L}\int_0^L (g(x)+\alpha)\sin(k_n x)\,dx.
\]

Defining
\[
\lambda_n=\frac{\beta k_n^4}{1+\alpha k_n^2},
\qquad
b_n=\frac{\gamma_n}{1+\alpha k_n^2},
\]
the modal system takes the normalized form
\begin{equation}\label{eq:rom_normalized}
	\ddot q_n(t)+\lambda_n q_n(t)=b_n\,\ddot u(t),
	\qquad n=1,\dots,N.
\end{equation}
The coefficients $b_n$ encode how the boundary actuation
couples with each vibration mode. In particular, their decay with
respect to the modal index $n$ introduces an intrinsic limitation
in the controllability of higher-frequency modes, which will play
a key role in the numerical observations of Section 6.

\subsection*{Choice of Actuation Variable}
\label{subsec:actuation_variable}

After the lifting transformation, the controlled beam dynamics take the form
\eqref{eq:lifted_w}
where $u(t)$ denotes the boundary curvature input and $g(x)$ is the lifting profile. Equation \eqref{eq:lifted_w} shows that, due to the inertial operator
\(
I-\alpha\partial_{xx},
\)
the system is naturally forced by the \emph{second time derivative} of the boundary signal.
Consequently, the effective control input entering the reduced-order dynamics is not $u(t)$ itself, but rather
\begin{equation}
	\label{eq:actuation_variable}
	a(t):=\ddot{u}(t).
\end{equation}

This choice is not merely a modeling convenience. It plays a structural role for three reasons:

\begin{enumerate}
	\item It preserves the intrinsic second-order mechanical structure of the beam dynamics.
	\item It ensures that the reduced-order model obtained via Galerkin projection remains Markovian.
	\item It provides a direct and consistent interface between the continuous PDE formulation and the finite-dimensional control input used in the reinforcement learning framework.
\end{enumerate}

In the sequel, we therefore treat $a(t)$ as the control action.

\section{Control and MDP formulation}

In this section, we formulate the control problem associated with the
reduced-order model obtained via lifting and Galerkin projection. We then recast
the system as a Markov decision process (MDP) to enable the application of
reinforcement learning methods.

\subsection{State-space formulation}

Let $q(t)=(q_1(t),\dots,q_N(t))^\top\in\mathbb{R}^N$ and let
$\Lambda=\mathrm{diag}(\lambda_1,\dots,\lambda_N)$ be the diagonal matrix of
modal coefficients, with $b=(b_1,\dots,b_N)^\top$.
We introduce the first-order modal state
\[
x(t):=\begin{bmatrix}q(t)\\ \dot q(t)\end{bmatrix}\in\mathbb{R}^{2N}.
\]
Then \eqref{eq:rom_normalized} is equivalent to the linear control system
\begin{equation}
	\dot x(t)=A x(t)+G\,a(t),
	\label{eq:ss_cont}
\end{equation}
where
\begin{equation}
	A=
	\begin{bmatrix}
		0 & I\\
		-\Lambda & 0
	\end{bmatrix},
	\qquad
	G=
	\begin{bmatrix}
		0\\
		b
	\end{bmatrix}.
	\label{eq:A_G}
\end{equation}

An approximate reconstruction of the continuous field is given by
\[
\varphi_N(x,t)=\sum_{n=1}^N q_n(t)\phi_n(x)-u(t)\,g(x),
\qquad
u(t)=u(0)+\dot u(0)t+\int_0^t\int_0^\tau a(s)\,ds\,d\tau,
\]
where $a(t)$ is interpreted as the (scalar) actuation variable associated with
the boundary input, and $g(x)$ is the lifting profile used to homogenize the
boundary conditions.

The reduced modal energy is
\begin{equation}
	\mathcal{E}_N(t)
	=
	\frac12\Big(\|\dot q(t)\|^2 + q(t)^\top\Lambda q(t)\Big).
	\label{eq:energy_modal}
\end{equation}
Along any sufficiently smooth trajectory of \eqref{eq:ss_cont}. 

\subsection{Optimal control problem}

Let $T>0$ be fixed. We consider the energy--effort cost functional
\begin{equation}
	J(a)=\int_0^T \Big( c_E\,\mathcal{E}_N(x(t)) + \mu\,a(t)^2 \Big)\,dt,
	\qquad c_E>0,\;\mu>0.
	\label{eq:cost_cont}
\end{equation}
The continuous-time optimal control problem is
\begin{equation}
	\min_{a\in\mathcal{U}} J(a)
	\quad\text{subject to}\quad
	\dot x=Ax+Ga,\;\; x(0)=x_0,
	\label{eq:OCP}
\end{equation}
where $\mathcal{U}$ is an admissible set, e.g., $\mathcal{U}=L^2(0,T)$ with
amplitude constraints $|a(t)|\le a_{\max}$ a.e.

The functional \eqref{eq:cost_cont} penalizes both the reduced vibrational
energy and the control effort. Minimizing $\int_0^T \mathcal{E}_N(t)\,dt$
encourages trajectories with small kinetic and potential modal energies,
while the term $\mu a(t)^2$ regularizes the actuation and prevents excessively
large control inputs. This structure is also convenient for reinforcement
learning, since the reward can be defined as the negative running cost.
For further results on existence/regularity of optimal controls for related
second-order systems, see, e.g., \cite{LIONS1969,LAGNESE-LINOS,daSilvaRodrigues2025}.

\subsection{Time discretization and transition operator}

Fix a time step $\Delta t>0$ and define $t_k=k\Delta t$, $k=0,1,\dots,K$, with
$K\Delta t=T$. Let $x_k\approx x(t_k)$ and $a_k\approx a(t_k)$. In our numerical
experiments, the transition map is generated by a fourth-order Runge--Kutta (RK4)
integrator applied to the continuous-time state equation \eqref{eq:ss_cont}, leading to a deterministic update of the form
\begin{equation}
	x_{k+1}=F_\Delta(x_k,a_k).
	\label{eq:transition_general}
\end{equation}
For simplicity of exposition, one may also consider an explicit Euler step,
$x_{k+1}=x_k+\Delta t(Ax_k+Ga_k)$, without changing the MDP structure.

\subsection{Closed-loop well-posedness of the discrete-time evaluation dynamics}
\label{subsec:closed_loop_wellposed}

During evaluation we apply the deterministic feedback $a_k=\mu_\theta(s_k)$, hence
\begin{equation}\label{eq:closed_loop_map}
	s_{k+1}=\mathcal{F}_\Delta(s_k):=F_\Delta\big(s_k,\mu_\theta(s_k)\big).
\end{equation}

\begin{assumption}[Local Lipschitz simulator and policy]
	\label{ass:lipschitz_both}
	On every bounded set, $F_\Delta(\cdot,\cdot)$ is locally Lipschitz in $(s,a)$ and
	$\mu_\theta(\cdot)$ is locally Lipschitz in $s$.
\end{assumption}

\begin{Proposition}[Well-posedness of the closed-loop recursion]
	\label{prop:wellposed_closedloop}
	Under Assumption~\ref{ass:lipschitz_both}, $\mathcal{F}_\Delta$ is globally Lipschitz on compact subsets of $\mathbb{R}^{2N}$
	In particular, for any initial condition $s_0\in\mathbb{R}^{2N}$ the recursion
	$s_{k+1}=\mathcal{F}_\Delta(s_k)$ defines a unique trajectory for $k=0,1,\dots,K$.
\end{Proposition}

\begin{proof}
	For $s,\tilde s$ in a bounded set, use Lipschitz continuity of $F_\Delta$ and $\mu_\theta$:
	\[
	\|\mathcal{F}_\Delta(s)-\mathcal{F}_\Delta(\tilde s)\|
	\le
	L\big(\|s-\tilde s\|+|\mu_\theta(s)-\mu_\theta(\tilde s)|\big)
	\le
	L(1+L_\mu)\|s-\tilde s\|.
	\]
	Uniqueness follows from determinism of the recursion. \end{proof}

\subsection{Formulation as a Markov decision process (MDP)}

\begin{Definition}
	The discrete-time MDP associated with \eqref{eq:transition_general} is
	\[
	\mathcal{M}=(\mathcal{S},\mathcal{A},P,r,\gamma),
	\]
	where:
	\begin{itemize}
		\item \textbf{State space:} $\mathcal{S}\subset\mathbb{R}^{2N}$ with
		$s_k:=x_k=(q_k,\dot q_k)$.
		\item \textbf{Action space:} $\mathcal{A}\subset\mathbb{R}$ with bounded actions
		$a_k\in[-a_{\max},a_{\max}]$.
		\item \textbf{Transition kernel:} deterministic,
		\[
		P(\cdot\,|\,s_k,a_k)=\delta_{F_\Delta(s_k,a_k)}(\cdot).
		\]
		\item \textbf{Reward:} negative running cost,
		\begin{equation}
			r(s_k,a_k)
			=
			-
			\Big(
			c_E\,\mathcal{E}_N(s_k)
			+
			\mu\,a_k^2
			\Big),
			\label{eq:reward}
		\end{equation}
		with $\mathcal{E}_N$ given by \eqref{eq:energy_modal}.
		\item \textbf{Discount factor:} $\gamma\in(0,1]$, which can be interpreted as a standard discounted surrogate of a
		finite-horizon objective.
	\end{itemize}
\end{Definition}

The reward function defined in \eqref{eq:reward} promotes
attenuation of the modal energy together with control regularization.
However, the reward itself does not constitute a Lyapunov certificate
for the closed-loop system. Nevertheless, when the learned feedback
exhibits an active damping structure along the actuated velocity channel
\(v=b^\top\dot q\), conditional energy decay can be established.

The RL objective is to find a stochastic policy $\pi_\theta(a|s)$ such that,
for actions sampled as $a_k \sim \pi_\theta(\cdot \mid s_k)$,
the expected return
\begin{equation}\label{eq:rl_objective}
\mathbb{E}_{\pi_\theta}\!\left[\sum_{k=0}^{K-1}\gamma^k r(s_k,a_k)\right]
\end{equation}
is maximized. The expectation $\mathbb{E}_{\pi_\theta}$ is taken with respect to the trajectory
distribution induced by the stochastic policy $\pi_\theta(a|s)$, i.e.,
actions are sampled as $a_k \sim \pi_\theta(\cdot \mid s_k)$ during training.

\begin{remark}[Connection with optimal control]
	In the deterministic setting, with $\gamma \approx 1$ and a consistent time
	discretization, maximizing \eqref{eq:rl_objective} with reward \eqref{eq:reward}
	is equivalent to minimizing a discrete-time approximation of the continuous
	cost functional \eqref{eq:cost_cont}. In this sense, reinforcement learning
	can be interpreted as a data-driven procedure to approximate an optimal
	feedback law without explicitly solving the Hamilton--Jacobi--Bellman equation.
\end{remark}

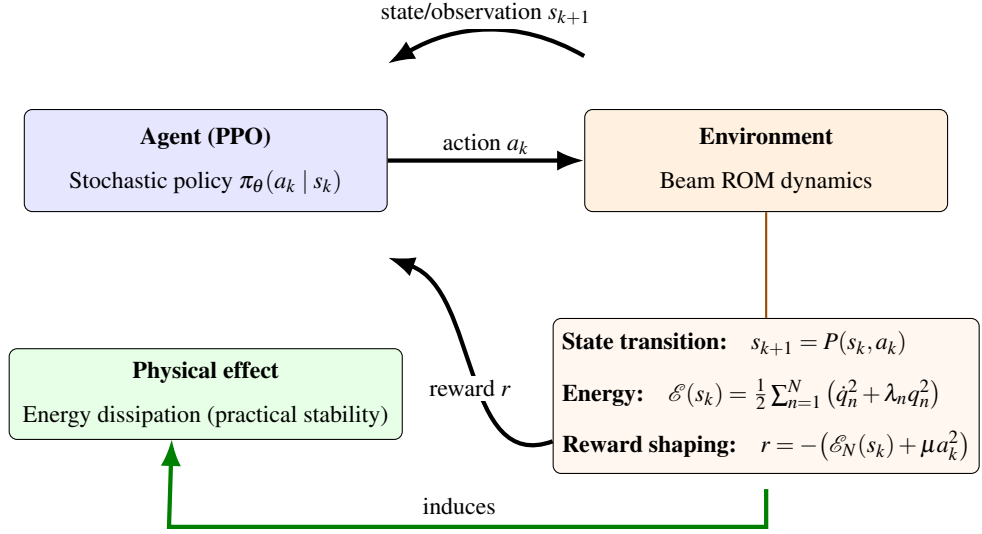
\begin{figure}[ht]
	\centering
	\begin{tikzpicture}[
		>=Latex,
		node distance=42mm and 26mm,
		arrow/.style={->, ultra thick},
		lab/.style={font=\small, fill=white, inner sep=1.5pt},
		agentbox/.style={
			draw, rounded corners=3pt, align=center,
			minimum width=4.8cm, minimum height=1.35cm,
			font=\small, fill=blue!10
		},
		envbox/.style={
			draw, rounded corners=3pt, align=center,
			minimum width=4.8cm, minimum height=1.35cm,
			font=\small, fill=orange!12
		},
		inner/.style={
			draw, rounded corners=3pt, align=left,
			minimum width=5.2cm, minimum height=2.05cm,
			font=\small, fill=orange!6
		},
		physbox/.style={
			draw, rounded corners=3pt, align=center,
			minimum width=5.2cm, minimum height=1.2cm,
			font=\small, fill=green!10
		}
		]
		
		\node[agentbox] (agent) {\textbf{Agent (PPO)}\\
			Stochastic policy $\pi_\theta(a_k\mid s_k)$};
		
		\node[envbox, right=of agent] (env) {\textbf{Environment}\\
			Beam ROM dynamics};
		
		\node[inner, below=14mm of env] (envinner) {
			\textbf{State transition:}\quad $s_{k+1}=P(s_k,a_k)$\\[1mm]
			\textbf{Energy:}\quad $\mathcal{E}(s_k)=\frac12\sum_{n=1}^{N}\big(\dot q_n^2+\lambda_n q_n^2\big)$\\[1mm]
			\textbf{Reward shaping:}\quad $r=-\big(\mathcal{E}_N(s_k)+\mu a_k^2\big)$
		};
		
		\draw[-, thick, draw=orange!60!black] (env.south) -- (envinner.north);
		
		\node[physbox, below=18mm of agent] (phys)
		{\textbf{Physical effect}\\
			Energy dissipation (practical stability)};
		
		\draw[arrow, draw=green!50!black]
		([yshift=-2mm]envinner.south)
		-- ++(0,-5mm)                      
		-- ++(-79mm,0)                      
		-- ([xshift=-4.5mm,yshift=-12mm]phys.north)       
		
		node[lab, pos=1.1,xshift=38mm, yshift=-10mm] {induces};
		
		
		\draw[arrow] (agent.east) -- node[lab, above] {action $a_k$} (env.west);
		
		\draw[arrow]
		([yshift=7mm]env.north west)
		to[out=150,in=30]
		node[lab, above] {state/observation $s_{k+1}$}
		([yshift=6mm]agent.north east);
		
		\draw[arrow]
		([yshift=-6mm]envinner.west)
		to[out=200,in=-20]
		node[lab, below] {reward $r$}
		([yshift=-6mm]agent.south east);
		
	\end{tikzpicture}
	
	\caption{
		Agent--environment interaction loop for boundary control via reinforcement learning.
		The agent (PPO) outputs the boundary action $a_k$, the environment propagates the reduced-order beam dynamics to $s_{k+1}$, and the reward is computed via an energy-based shaping $r=-\big(\mathcal{E}_N(s_k)+\mu a_k^2\big)$, promoting dissipative behavior and practical stability.
	}
	\label{fig:rl_agent_environment_loop}
\end{figure}


\begin{Theorem}	\label{Theo:1}
	Consider a finite horizon $T>0$ and a consistent time discretization with step
	$\Delta t=T/K$, producing a discrete trajectory $(s_k,a_k)_{k=0}^{K-1}$.
	Define the instantaneous cost
	\[
	\ell(s_k,a_k):=\mathcal{E}_N(s_k)+\mu a_k^2,
	\qquad \mu>0,
	\]
	and the reward
	\[
	r(s_k,a_k):=-\ell(s_k,a_k).
	\]
	For $\gamma=1$, maximizing the expected return
	\[
	\max_{\pi_\theta}\;
	\mathbb{E}_{\pi_\theta}\!\left[\sum_{k=0}^{K-1} r(s_k,a_k)\right]
	\]
	is equivalent, up to positive constants, to minimizing the accumulated discrete cost
	\[
	\mathbb{E}_{\pi_\theta}\!\left[\sum_{k=0}^{K-1} \ell(s_k,a_k)\right],
	\]
	which approximates the continuous functional $J(u)$ as $\Delta t\to 0$.
\end{Theorem}

\begin{proof}
	By definition of the reward, for any state--action pair $(s_k,a_k)$,
	\[
	r(s_k,a_k)=-\ell(s_k,a_k).
	\]
	Since $\gamma=1$, the finite-horizon return associated with a trajectory
	$(s_k,a_k)_{k=0}^{K-1}$ is
	\[
	\sum_{k=0}^{K-1} r(s_k,a_k)= -\sum_{k=0}^{K-1}\ell(s_k,a_k).
	\]
	Taking expectation under the policy $\pi_\theta$ and the trajectory distribution
	induced by the environment dynamics yields
	\[
	\mathbb{E}_{\pi_\theta}\!\left[\sum_{k=0}^{K-1} r(s_k,a_k)\right]
	=
	-\mathbb{E}_{\pi_\theta}\!\left[\sum_{k=0}^{K-1}\ell(s_k,a_k)\right].
	\]
	Since the map $x\mapsto -x$ is strictly monotone, it follows that
	\[
	\arg\max_{\pi_\theta}\;
	\mathbb{E}_{\pi_\theta}\!\left[\sum_{k=0}^{K-1} r(s_k,a_k)\right]
	=
	\arg\min_{\pi_\theta}\;
	\mathbb{E}_{\pi_\theta}\!\left[\sum_{k=0}^{K-1}\ell(s_k,a_k)\right].
	\]
	That is, maximizing return is equivalent to minimizing the accumulated discrete cost
	(the equivalence is exact; in particular, multiplying by any positive constant does
	not change the argmax/argmin).
	
	To relate the discrete cost to the continuous functional, consider a zero-order hold
	control $a_k$ on $[t_k,t_{k+1})$, where $t_k=k\Delta t$. Define
	\[
	J(a)=\int_0^T \big(\mathcal{E}_N(x(t))+\mu a(t)^2\big)\,dt,
	\]
	and its Riemann-sum approximation
	\[
	J_{\Delta t}(a):=\sum_{k=0}^{K-1}\big(\mathcal{E}_N(x(t_k))+\mu a_k^2\big)\,\Delta t.
	\]
	Under consistency of the integrator and
	continuity of $\mathcal{E}_N$, we obtain
	\[
	\sum_{k=0}^{K-1}\ell(s_k,a_k)\,\Delta t
	=
	\sum_{k=0}^{K-1}\big(\mathcal{E}_N(x(t_k))+\mu a_k^2\big)\,\Delta t + O(\Delta^4)
	=
	J_{\Delta t}(a)+O(\Delta^4).
	\]
	Since $J_{\Delta t}(a)$ is a Riemann sum for $J(a)$, it follows that
	$J_{\Delta t}(a)\to J(a)$ as $\Delta t\to0$. Therefore,
	\[
	\sum_{k=0}^{K-1}\ell(s_k,a_k)\,\Delta t \longrightarrow J(a),
	\qquad \Delta t\to0,
	\]
	showing that the discrete cost  approximates the continuous
	functional. Assuming the cost function is bounded and the policy induces measurable trajectory distributions, the expectation of the Riemann sum converges to the expectation of the functional by the Lebesgue Dominated Convergence Theorem \cite{EVANS2010}.
\end{proof}

\section{Proximal Policy Optimization (PPO) algorithm}

In this section, we describe the Proximal Policy Optimization (PPO) algorithm
\cite{Schulman2017PPO} employed to solve the Markov decision process associated
with the reduced-order beam model. In addition to the theoretical formulation, we
explicitly document the implementation choices adopted in the numerical
experiments (actor--critic architecture, GAE, clipping, minibatch training, and
observation normalization), ensuring full consistency between algorithmic
description and computational results.
\subsection{Parameterized stochastic policy}

Consider the MDP defined in the previous section, with states
$s_k=(q_k,\dot q_k)\in\mathbb{R}^{2N}$ and actions $a_k\in\mathbb{R}$.
The policy is modeled as a family of parameterized probability distributions
\[
\pi_\theta(a|s), \qquad \theta\in\mathbb{R}^p,
\]
where $\theta$ denotes the parameters of a neural network.

In this work, we adopt a one-dimensional Gaussian policy,
\begin{equation}\label{eq:policy_gaussian}
	\pi_\theta(a|s)
	=
	\frac{1}{\sqrt{2\pi\sigma_\theta^2}}
	\exp\!\left(
	-\frac{(a-\mu_\theta(s))^2}{2\sigma_\theta^2}
	\right),
\end{equation}
where the mean $\mu_\theta:\mathbb{R}^{2N}\to\mathbb{R}$ is represented by a
multilayer perceptron with two hidden layers of 64 neurons each and hyperbolic
tangent activation functions. The variance $\sigma_\theta>0$  is state-independent and is parameterized via a trainable log-standard deviation, as in \texttt{stable-baselines3}.

More precisely, the policy mean is given by
\begin{equation}
	\mu_\theta(s)
	=
	W_3\,\tanh\!\Big(
	W_2\,\tanh(W_1 s + b_1) + b_2
	\Big) + b_3,
	\label{eq:actor_network}
\end{equation}
where
\[
W_1\in\mathbb{R}^{64\times 2N}, \quad
W_2\in\mathbb{R}^{64\times 64}, \quad
W_3\in\mathbb{R}^{1\times 64},
\]
and $b_1,b_2,b_3$ are bias vectors of compatible dimensions.

The choice of $\tanh$ activation functions ensures smoothness of the policy and
naturally bounds intermediate representations, which is desirable in physical
control applications.


\usetikzlibrary{positioning,arrows.meta}

\begin{figure}[t]
	\centering
	\begin{tikzpicture}[
		>=Latex,
		neuron/.style={circle, draw, thick, minimum size=6mm},
		input/.style={neuron, fill=blue!20},
		hidden/.style={neuron, fill=green!20},
		output/.style={neuron, fill=red!20},
		layerlabel/.style={font=\small},
		connect/.style={->, thick}
		]
		
		\node[input, yshift=-4mm] (q) {$q$};
		\node[input, below=6mm of q, yshift=-4mm] (qdot) {$\dot q$};
		
		\node[layerlabel, above=4mm of q,yshift=-2mm] {Input layer};
		\node[layerlabel, below=2mm of qdot,yshift=-2mm] {$s=(q,\dot q)$};
		
		\node[hidden, right=20mm of q, yshift=5mm] (h11) {};
		\node[hidden, below=4mm of h11] (h12) {};
		\node[hidden, below=12mm of h12] (h13) {};
		\node[hidden, below=4mm of h13] (h14) {};
		
		\foreach \t in {0.25,0.5,0.75} {
			\node[circle, fill=black, inner sep=0pt, minimum size=2.5pt]
			at ($(h12)!{\t}!(h13) + (0mm,0)$) {};
		}
		
		\node[layerlabel, above=4mm of h11] {Hidden layer 1};
		\node[layerlabel, below=2mm of h14] {$64$ neurons, $\tanh$};
		
		\node[hidden, right=20mm of h11, yshift=0mm] (h21) {};
		\node[hidden, below=4mm of h21] (h22) {};
		\node[hidden, below=12mm of h22] (h23) {};
		\node[hidden, below=4mm of h23] (h24) {};
		
		\foreach \t in {0.25,0.5,0.75} {
			\node[circle, fill=black, inner sep=0pt, minimum size=2.5pt]
			at ($(h12)!{\t}!(h13) + (26mm,0)$) {};
		}
		\node[layerlabel, above=4mm of h21] {Hidden layer 2};
		\node[layerlabel, below=2mm of h24] {$64$ neurons, $\tanh$};
		
		\node[output, right=22mm of h21, yshift=-16mm] (out) {$\mu_\theta(s)$};
		
		\node[layerlabel, above=4mm of out] {Output layer};
		\node[layerlabel, below=2mm of out] {Action mean};
		
		\foreach \h in {h11,h12,h13,h14} {
			\draw[connect] (q) -- (\h);
			\draw[connect] (qdot) -- (\h);
		}
		
		\foreach \hA in {h11,h12,h13,h14} {
			\foreach \hB in {h21,h22,h23,h24} {
				\draw[connect] (\hA) -- (\hB);
			}
		}
		
		\foreach \h in {h21,h22,h23,h24} {
			\draw[connect] (\h) -- (out);
		}
		
	\end{tikzpicture}
	
	\caption{Didactic illustration of the multilayer perceptron (MLP) used to parametrize the PPO policy.
		Only a few neurons per layer are shown for clarity. The actual network has two hidden layers with
		64 neurons each and $\tanh$ activation functions.}
	\label{fig:mlp_didactic}
\end{figure}
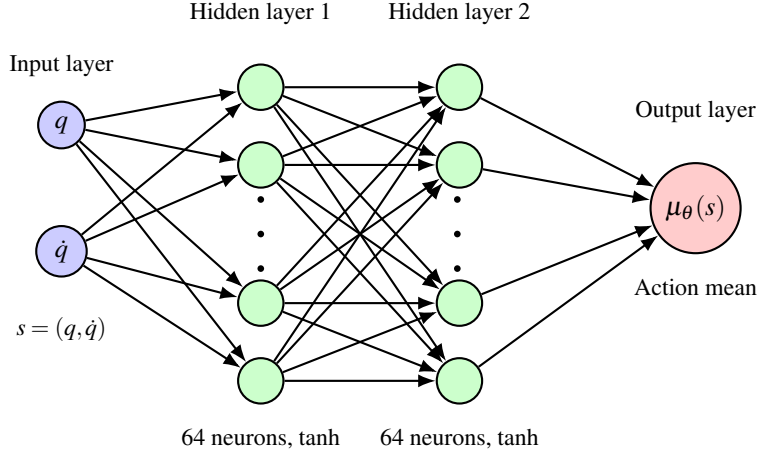
The Gaussian policy is commonly adopted in continuous-control problems,
as it ensures differentiability of the log-likelihood with respect to the
policy parameters. This property is essential for policy-gradient methods,
since the gradient of the objective involves $\nabla_\theta \log
\pi_\theta(a|s)$. Moreover, the stochasticity of the Gaussian distribution
facilitates exploration of the action space during training
\cite{Schulman2017PPO,SuttonBarto2018}.



\subsection{Value function and advantage function}

The value function associated with a policy $\pi_\theta$ is defined by
\begin{equation}
	V^{\pi_\theta}(s)
	=
	\mathbb{E}_{\pi_\theta}
	\left[
	\sum_{k=0}^{K-1}\gamma^k\,r(s_k,a_k)
	\;\middle|\;
	s_0=s
	\right],
\end{equation}
where $\gamma\in(0,1]$ is the discount factor.

In PPO, the value function is approximated by a neural network
$V_\phi:\mathbb{R}^{2N}\to\mathbb{R}$, called the \emph{critic}. The critic shares
the same architecture as the actor, namely two hidden layers of 64 neurons with
$\tanh$ activation functions, but has an independent set of parameters $\phi$. In our implementation, the critic is trained jointly with the actor through a squared-error
regression loss against empirical returns computed from rollouts, using the same minibatch
sampling and multiple-epoch updates as the actor.

\begin{Lemma}[Baseline invariance]
	\label{lem:baseline_invariance}
	Let $b:\mathcal{S}\to\mathbb{R}$ be measurable and independent of $a$. If
	$\pi_\theta(\cdot\mid s)$ is differentiable and $\nabla_\theta\pi_\theta(\cdot\mid s)$
	is integrable, then for each fixed $s$,
	\[
	\mathbb{E}_{a\sim\pi_\theta(\cdot\mid s)}\!\left[\nabla_\theta\log\pi_\theta(a\mid s)\,b(s)\right]=0.
	\]
\end{Lemma}

\begin{proof}
	Compute
	\[
	\mathbb{E}\big[\nabla_\theta\log\pi_\theta(a\mid s)\big]
	=
	\int \nabla_\theta\log\pi_\theta(a\mid s)\,\pi_\theta(a\mid s)\,da
	=
	\int \nabla_\theta\pi_\theta(a\mid s)\,da
	=
	\nabla_\theta\int \pi_\theta(a\mid s)\,da
	=
	\nabla_\theta 1
	=
	0,
	\]
	and multiply by $b(s)$.
\end{proof}

\subsection{Advantage function and generalized advantage estimation}

The action--value function is defined by
\[
Q^{\pi_\theta}(s,a)
=
\mathbb{E}_{\pi_\theta}
\left[
\sum_{k=0}^{K-1}
\gamma^k\,r(s_k,a_k)
\;\middle|\;
s_0=s,\; a_0=a
\right].
\]

\begin{Definition}[Advantage function {\cite{SuttonBarto2018}}]
	The advantage function is given by
	\[
	A^{\pi_\theta}(s,a)
	=
	Q^{\pi_\theta}(s,a)
	-
	V^{\pi_\theta}(s),
	\]
	and measures the relative benefit of selecting action $a$ in state $s$.
\end{Definition}

The advantage function measures the relative gain of taking action $a$ in state
$s$. In practice, $V^{\pi_\theta}$ is approximated by a parametric function
$V_\phi$, called the \emph{critic}, as in the actor--critic architecture
\cite{KondaTsitsiklis2000}.

\begin{Theorem}[Policy Gradient Theorem (trajectory-density form) {\cite{SuttonBarto2018}}]
	\label{thm:policy_gradient}
	Assume $P$ does not depend on $\theta$, and $\pi_\theta$ is differentiable with sufficient
	integrability to interchange gradient and expectation. Then
	\[
	\nabla_\theta J(\theta)
	=
	\mathbb{E}_{\pi_\theta}\big[\nabla_\theta\log\pi_\theta(a_k\mid s_k)\,Q^{\pi_\theta}(s_k,a_k)\big].
	\]
	Moreover, the advantage form holds:
	\[
	\nabla_\theta J(\theta)
	=
	\mathbb{E}_{\pi_\theta}\big[\nabla_\theta\log\pi_\theta(a_k\mid s_k)\,A^{\pi_\theta}(s_k,a_k)\big].
	\]
\end{Theorem}

\begin{proof}
	The first identity follows from the likelihood-ratio (score-function) derivation on the
	trajectory density induced by $\pi_\theta$ (standard in \cite{SuttonBarto2018}).
	For the second identity, use Lemma~\ref{lem:baseline_invariance} with baseline
	$b(s)=V^{\pi_\theta}(s)$ to replace $Q^{\pi_\theta}$ by $Q^{\pi_\theta}-V^{\pi_\theta}$,
	which equals $A^{\pi_\theta}$.
\end{proof}

To reduce the variance of the policy-gradient estimator, we use generalized
advantage estimation (GAE), proposed in \cite{Schulman2016GAE}.
Defining the temporal-difference error
\subsection{Generalized Advantage Estimation (GAE)}
\label{subsec:gae}

Since $Q^{\pi_\theta}$ and $V^{\pi_\theta}$ are unknown, we estimate advantages from rollouts.
We use generalized advantage estimation (GAE) \cite{Schulman2016GAE}. Define the TD-error
\begin{equation}\label{eq:td_error}
	\delta_k=r_k+\gamma V_\phi(s_{k+1})-V_\phi(s_k),
\end{equation}
and the GAE estimator
\begin{equation}\label{eq:gae}
	\widehat A_k=\sum_{l=0}^{K-k-1}(\gamma\lambda)^l\,\delta_{k+l},\qquad \lambda\in[0,1].
\end{equation}
We set $\gamma=0.99$ and $\lambda=0.95$ in all experiments. The critic regression
targets are computed as
\begin{equation}\label{eq:return_target}
	\widehat V_k=\widehat A_k+V_\phi(s_k).
\end{equation}

\subsection{Clipped surrogate objective and trust-region viewpoint}
\label{subsec:ppo_clip_trust}

Large policy updates may degrade performance; trust-region methods control policy drift,
typically via KL constraints. In TRPO-type formulations one considers
\begin{equation}\label{eq:trpo_constrained}
	\max_{\theta}\ \mathbb{E}\big[r_k(\theta)\,\widehat A_k\big]
	\quad \text{s.t.}\quad
	\mathbb{E}\big[D_{\mathrm{KL}}(\pi_{\theta_{\mathrm{old}}}(\cdot\mid s_k)\,\|\,\pi_{\theta}(\cdot\mid s_k))\big]\le \delta,
\end{equation}
where the likelihood ratio is
\begin{equation}\label{eq:ratio}
	r_k(\theta)=\frac{\pi_\theta(a_k\mid s_k)}{\pi_{\theta_{\mathrm{old}}}(a_k\mid s_k)}.
\end{equation}
PPO replaces the constrained problem \eqref{eq:trpo_constrained} by the clipped surrogate
\cite{Schulman2017PPO}
\begin{equation}\label{eq:ppo_clip}
	L^{\mathrm{CLIP}}(\theta)
	=
	\mathbb{E}
	\left[
	\min\Big(
	r_k(\theta)\widehat A_k,\;
	\mathrm{clip}(r_k(\theta),1-\varepsilon,1+\varepsilon)\widehat A_k
	\Big)
	\right],
\end{equation}
with $\varepsilon>0$.

The clipping operator restricts the probability ratio to the interval
$[1-\varepsilon,1+\varepsilon]$, preventing excessively large policy updates.
This mechanism stabilizes training and allows multiple epochs of minibatch
updates on the same rollout.
\subsection{Actor--critic loss and optimization details}
\label{subsec:ppo_loss_updates}

Actor and critic are trained jointly by minimizing
\begin{equation}\label{eq:ppo_total_loss}
	\mathcal{L}(\theta,\phi)
	=
	- L^{\mathrm{CLIP}}(\theta)
	+
	c_v\,\mathbb{E}\!\left[\big(V_\phi(s_k)-\widehat V_k\big)^2\right]
	-
	c_e\,\mathbb{E}\!\left[\mathcal{H}(\pi_\theta(\cdot\mid s_k))\right],
\end{equation}
where $c_v>0$ weights the critic regression term, $c_e\ge 0$ weights the entropy bonus, and
$\mathcal{H}(\pi(\cdot\mid s))=-\int \pi(a\mid s)\log\pi(a\mid s)\,da$.
Optimization is carried out with ADAM using minibatch SGD and multiple epochs per rollout,
and gradient norms are clipped to improve numerical stability.


\subsection{Algorithmic template and reproducibility notes}
\label{subsec:ppo_algorithm_template}

\begin{algorithm}[H]
	\caption{PPO (actor--critic with GAE and clipping) for beam control}
	\label{alg:ppo}
	\begin{algorithmic}[1]
		\STATE Initialize actor parameters $\theta$ and critic parameters $\phi$.
		\STATE Set $\theta_{\mathrm{old}}\leftarrow \theta$.
		\FOR{each training iteration}
		\STATE Collect rollout $\{(s_k,a_k,r_k)\}_{k=0}^{n_{\mathrm{steps}}-1}$ with $a_k\sim\pi_{\theta_{\mathrm{old}}}(\cdot\mid s_k)$.
		\STATE Compute $V_\phi(s_k)$ and TD-errors $\delta_k=r_k+\gamma V_\phi(s_{k+1})-V_\phi(s_k)$.
		\STATE Compute $\widehat A_k$ via \eqref{eq:gae} and returns $\widehat V_k=\widehat A_k+V_\phi(s_k)$.
		\FOR{$1,\dots,n_{\mathrm{epochs}}$}
		\STATE Shuffle and split rollout into minibatches.
		\FOR{each minibatch}
		\STATE Compute ratios $r_k(\theta)=\pi_\theta(a_k\mid s_k)/\pi_{\theta_{\mathrm{old}}}(a_k\mid s_k)$.
		\STATE Evaluate $L^{\mathrm{CLIP}}(\theta)$ via \eqref{eq:ppo_clip}.
		\STATE Update $(\theta,\phi)$ by minimizing \eqref{eq:ppo_total_loss} with ADAM.
		\STATE Apply gradient-norm clipping.
		\ENDFOR
		\ENDFOR
		\STATE $\theta_{\mathrm{old}}\leftarrow \theta$.
		\ENDFOR
	\end{algorithmic}
\end{algorithm}

\textbf{Observation normalization.}
Observations are normalized online using running mean/variance estimates
(\texttt{VecNormalize}). This improves conditioning across modal coordinates and velocities.

\textbf{Deterministic evaluation.}
Reported closed-loop results use the mean action $a=\mu_\theta(s)$,
yielding a deterministic nonlinear feedback law. Since the map
\(s \mapsto \mu_\theta(s)\) is a finite composition of affine maps
and the smooth activation function \(\tanh\), it is \(C^\infty\),
and therefore locally Lipschitz on bounded sets. Consequently,
Assumption~\ref{ass:lipschitz_both} is satisfied by the deterministic
evaluation policy used in the closed-loop analysis.

\section{Energy analysis and closed-loop stability}
\label{sec:energy_stability}

In this section, we analyze the energy structure of the reduced-order model and
the closed-loop behavior induced by different feedback strategies. We first
establish energy conservation in the uncontrolled case, then derive
dissipation mechanisms and structural limitations under bounded actuation,
and finally interpret the PPO controller within the same energetic framework.
These results provide the theoretical basis for the numerical comparisons
reported later.

\subsection{ROM modal energy and conservation (open loop)}

Consider the reduced-order model
\begin{equation}
	\ddot q(t)+\Lambda q(t)=b\,a(t),
	\qquad
	q(t)\in\mathbb{R}^N,\quad
	\Lambda=\mathrm{diag}(\lambda_1,\dots,\lambda_N),\quad
	b\in\mathbb{R}^N,
	\label{eq:rom_vec}
\end{equation}
where \(a(t)=\ddot u(t)\) is the control action. We define the reduced modal energy
\begin{equation}
	\mathcal{E}_N(q,\dot q)
	=
	\frac12\left(\|\dot q\|^2+q^\top\Lambda q\right).
	\label{eq:EN}
\end{equation}
Differentiating \eqref{eq:EN} along solutions of \eqref{eq:rom_vec} gives
\begin{align}
	\frac{d}{dt}\mathcal{E}_N(q,\dot q)
	&=
	\dot q^\top \ddot q+\dot q^\top\Lambda q
	=
	\dot q^\top(\ddot q+\Lambda q)
	=
	\dot q^\top b\,a(t).
	\label{eq:ENdot_open}
\end{align}
In particular, if \(a(t)\equiv 0\), then \eqref{eq:ENdot_open} yields
\[
\frac{d}{dt}\mathcal{E}_N(q,\dot q)=0,
\]
so that the reduced system is conservative in open loop.

\subsection{Dissipation mechanisms and structural limitations}

To introduce dissipation, consider the linear velocity feedback law
\begin{equation}
	a(t)=-K\,\dot q(t),
	\qquad
	K\in\mathbb{R}^{1\times N}.
	\label{eq:vel_feedback}
\end{equation}
Substituting \eqref{eq:vel_feedback} into \eqref{eq:ENdot_open}, we obtain
\begin{equation}
	\dot{\mathcal{E}}_N(q,\dot q)
	=
	-\dot q^\top (bK)\,\dot q.
	\label{eq:ENdot_closed_lin}
\end{equation}

\begin{Proposition}[Dissipation along the actuated velocity channel]
	\label{prop:colocated_damping}
	Consider the closed-loop feedback
	\[
	a(t)=-k\,b^\top \dot q(t),
	\qquad k\ge 0,
	\]
	where \(b\in\mathbb{R}^N\) is the input vector in \eqref{eq:rom_vec}.
	Then the modal energy satisfies
	\[
	\dot{\mathcal{E}}_N(t)
	=
	-k\,(b^\top \dot q(t))^2
	\le 0.
	\]
	In particular,
	\[
	\dot{\mathcal{E}}_N(t)=0
	\quad\Longleftrightarrow\quad
	b^\top \dot q(t)=0.
	\]
	Hence energy dissipation occurs exclusively through the scalar actuated
	velocity channel \(b^\top\dot q(t)\).
\end{Proposition}

\begin{proof}
	From the energy identity \eqref{eq:ENdot_open},
	\[
	\dot{\mathcal{E}}_N(t)=a(t)\,b^\top\dot q(t).
	\]
	Substituting the feedback law gives
	\[
	\dot{\mathcal{E}}_N(t)
	=
	(-k\,b^\top\dot q(t))\,b^\top\dot q(t)
	=
	-k\,(b^\top\dot q(t))^2.
	\]
	Since \(k\ge 0\), the right-hand side is non-positive, and
	\[
	\dot{\mathcal{E}}_N(t)=0
	\quad\Longleftrightarrow\quad
	(b^\top\dot q(t))^2=0
	\quad\Longleftrightarrow\quad
	b^\top\dot q(t)=0.
	\]
\end{proof}

\begin{remark}
	The feedback \eqref{eq:vel_feedback} generates dissipation only through the
	scalar velocity channel \(b^\top\dot q(t)\). Consequently, damping acts
	solely along the actuated modal direction \(\mathrm{span}\{b\}\). Since the
	control input is scalar, it cannot independently damp all modal velocity
	components; the remaining components may decay only through the intrinsic
	modal coupling of the system dynamics. This structure provides a natural
	energetic baseline against which more general nonlinear or learning-based
	feedback laws can be interpreted.
\end{remark}

\begin{Proposition}[Fundamental dissipation bound under saturated SISO actuation]
	\label{prop:dissipation_bound}
	Consider the reduced-order model \eqref{eq:rom_vec} with
	\(\Lambda\succ 0\) and \(b\in\mathbb{R}^N\). Assume that the control is bounded,
	\[
	|a(t)|\le a_{\max}.
	\]
	Then, along any trajectory generated by any admissible feedback law, the
	modal energy satisfies
	\[
	\frac{d}{dt}\mathcal{E}_N(t)
	\ge
	-a_{\max}\,\|b\|\,\sqrt{2\,\mathcal{E}_N(t)}.
	\]
	Consequently, for any \(\rho\in(0,1)\), the time required to reach
	\[
	\mathcal{E}_N(t_\rho)=\rho\,\mathcal{E}_N(0)
	\]
	is bounded below by
	\[
	t_\rho
	\ge
	\frac{\sqrt{2}}{a_{\max}\|b\|}
	\Big(
	\sqrt{\mathcal{E}_N(0)}
	-
	\sqrt{\rho\,\mathcal{E}_N(0)}
	\Big).
	\]
\end{Proposition}

\begin{proof}
	Differentiating the energy along solutions yields
	\[
	\dot{\mathcal{E}}_N(t)
	=
	\langle \dot q(t),\ddot q(t)+\Lambda q(t)\rangle
	=
	a(t)\,\langle b,\dot q(t)\rangle.
	\]
	By the Cauchy--Schwarz inequality,
	\[
	|\langle b,\dot q\rangle|
	\le
	\|b\|\,\|\dot q\|.
	\]
	Using the saturation constraint,
	\[
	\dot{\mathcal{E}}_N(t)
	\ge
	-a_{\max}\,\|b\|\,\|\dot q(t)\|.
	\]
	Since \(\|\dot q(t)\|^2\le 2\,\mathcal{E}_N(t)\), we obtain
	\[
	\dot{\mathcal{E}}_N(t)
	\ge
	-a_{\max}\,\|b\|\,\sqrt{2\,\mathcal{E}_N(t)}.
	\]
	Let \(x(t)=\sqrt{\mathcal{E}_N(t)}\). Then
	\[
	\dot x(t)
	=
	\frac{\dot{\mathcal{E}}_N(t)}{2\sqrt{\mathcal{E}_N(t)}}
	\ge
	-\frac{a_{\max}\|b\|}{\sqrt{2}}.
	\]
	Integration gives the stated lower bound for \(t_\rho\).
\end{proof}

\begin{corollary}[Structural degradation with increasing modal dimension]
	\label{cor:modal_degradation}
	Assume that the modal coupling coefficients \(b_n\) decay with the mode index,
	i.e.
	\[
	|b_n|\le C\,n^{-p},
	\qquad p>0.
	\]
	Consider initial data concentrated in a single high mode \(j\). Then the
	minimal time required to dissipate a fixed fraction of modal energy satisfies
	the lower bound
	\[
	t_\rho^{(j)}
	\gtrsim
	\frac{1}{a_{\max}|b_j|}
	\sim
	j^p.
	\]
	In particular, if \(b_j\) decays polynomially, the achievable dissipation rate
	necessarily deteriorates with increasing modal index.
\end{corollary}

\subsection{Baseline LQR and exponential stability of the linear system}

Define the first-order state \(x=(q,\dot q)\in\mathbb{R}^{2N}\) and the linear system
\begin{equation}
	\dot x = A x + G a,
	\qquad
	A=\begin{bmatrix}0&I\\ -\Lambda&0\end{bmatrix},
	\quad
	G=\begin{bmatrix}0\\ b\end{bmatrix}.
	\label{eq:ss_lqr}
\end{equation}
Consider the quadratic cost functional
\begin{equation}
	J(a)=\int_0^\infty \left(x^\top Q x + a^\top R a\right)\,dt,
	\qquad Q\succeq 0,\; R\succ 0.
	\label{eq:lqr_cost}
\end{equation}
If the pair \((A,G)\) is stabilizable and \((Q^{1/2},A)\) is detectable, then
there exists a unique matrix \(P\succeq 0\) solving the Algebraic Riccati Equation
\begin{equation}
	A^\top P + P A - P G R^{-1} G^\top P + Q = 0,
	\label{eq:ARE}
\end{equation}
and the optimal closed-loop law is
\begin{equation}
	a(t) = -R^{-1}G^\top P\,x(t) =: -K_{\mathrm{LQR}}x(t),
	\label{eq:lqr_law}
\end{equation}
which ensures exponential stability of the unconstrained closed-loop system
\(\dot x=(A-GK_{\mathrm{LQR}})x\).

\begin{remark}[Comparison under saturation]
	In the numerical implementation, the LQR controller is evaluated with the
	same amplitude constraint \(a\in[-a_{\max},a_{\max}]\) as PPO. Under this
	saturation, the resulting controller is no longer optimal for the
	unconstrained quadratic problem \eqref{eq:lqr_cost}, but it provides a
	standard deterministic baseline under actuator limits.
\end{remark}

\subsection{Interpretation of PPO as effective active damping}

In the considered MDP, the reward is
\[
r(s,a)=-(\mathcal{E}_N(s)+\mu a^2).
\]
Hence maximizing the return is equivalent to minimizing a weighted
combination of modal energy and control effort. Although this objective
does not provide a Lyapunov certificate by itself, it promotes feedback
policies that reduce energy along closed-loop trajectories.

From this viewpoint, PPO may be interpreted as a data-driven mechanism
for learning nonlinear state-dependent feedback laws that act
dissipatively on the reduced-order system. This interpretation becomes
more explicit under additional structural conditions on the learned
policy, as discussed below.

\subsection{Practical stability: saturation, penalization, and clipping}

The numerical implementation enforces an amplitude constraint
\[
|a(t)|\le a_{\max}.
\]
Moreover, the PPO objective includes the penalization term \(\mu a^2\),
which discourages large-magnitude actions.

PPO clipping constrains the policy \emph{update} in distributional space
and should not be confused with action saturation, which constrains the
\emph{applied control}. Both mechanisms contribute to stability in
practice: clipping stabilizes learning, while saturation enforces
physical feasibility.

A sufficient local condition for instantaneous dissipation is to choose
the action with sign opposite to the quantity \(\dot q^\top b\) in
\eqref{eq:ENdot_open}. In particular, if
\begin{equation}
	a(t) = -k\,\mathrm{sat}\big(\dot q(t)^\top b\big),
	\qquad k>0,
	\label{eq:energy_shaping}
\end{equation}
where \(\mathrm{sat}(\cdot)\) is a scalar saturation, then
\[
\dot{\mathcal{E}}_N(t)=\dot q^\top b\,a(t)\le 0,
\]
with strict inequality whenever \(\dot q^\top b\neq 0\) and the saturation
does not annihilate the product.

\begin{remark}[Connection with PPO]
	Although the damping structure \eqref{eq:energy_shaping} is not imposed
	explicitly, the reward formulation encourages policies that reduce the
	velocity channel \(b^\top\dot q\). As observed in the numerical
	experiments, the PPO policy tends to generate state-dependent
	damping-like behavior along this channel, leading to effective energy
	dissipation.
\end{remark}

The equivalence established in Theorem~\eqref{Theo:1} shows that the
reinforcement-learning formulation is consistent with the underlying
optimal control objective. The following proposition provides a
sufficient condition under which the learned feedback can be interpreted
as an effective nonlinear damping mechanism on the actuated channel.

Under the feedback law $a(t)=\mu_\theta(q(t),\dot q(t))$, the closed-loop
dynamics become
\[
\ddot q(t)+\Lambda q(t)=b\,\mu_\theta(q(t),\dot q(t)).
\]

\begin{Proposition}[Emergent nonlinear active damping mechanism]
	\label{prop:active_damping}
	Consider the modal system \eqref{eq:rom_vec} with modal energy
	\(\mathcal{E}_N\) and power channel
	\[
	v(t)=b^\top \dot q(t).
	\]
	If the learned policy $a(t)=\mu_\theta(q(t),\dot q(t))$ satisfies
	the sector condition
	\[
	\mu_\theta(q,\dot q)\,v \le -\eta v^2
	\]
	for some \(\eta>0\) in the regime under consideration, then
	\[
	\dot{\mathcal{E}}_N(t)\le -\eta v(t)^2\le 0.
	\]
	Thus, the learned policy acts as a nonlinear active damper on the
	actuated velocity channel.
\end{Proposition}

\begin{proof}
	We divide the proof into three steps.
	
	\medskip
	\noindent\textbf{Step 1: Power balance identity.}
	Differentiating \eqref{eq:EN} along classical trajectories of
	\eqref{eq:rom_vec} gives
	\[
	\dot{\mathcal{E}}_N(t)
	=
	\dot q(t)^\top \ddot q(t)+q(t)^\top \Lambda \dot q(t).
	\]
	Using \(\ddot q(t)=-\Lambda q(t)+b\,a(t)\), we obtain
	\[
	\dot{\mathcal{E}}_N(t)
	=
	\dot q(t)^\top\bigl(-\Lambda q(t)+b\,a(t)\bigr)
	+
	q(t)^\top \Lambda \dot q(t).
	\]
	Since \(\Lambda\) is symmetric, the conservative terms cancel, and hence
	\[
	\dot{\mathcal{E}}_N(t)=a(t)\,b^\top\dot q(t)=a(t)\,v(t).
	\]
	
	\medskip
	\noindent\textbf{Step 2: Dissipation under the sector condition.}
	Under the learned feedback \(a(t)=\mu_\theta(q(t),\dot q(t))\), the
	previous identity becomes
	\[
	\dot{\mathcal{E}}_N(t)=\mu_\theta(q(t),\dot q(t))\,v(t).
	\]
	Therefore, by the assumed sector condition,
	\[
	\dot{\mathcal{E}}_N(t)\le -\eta\,v(t)^2\le 0.
	\]
	
	\medskip
	\noindent\textbf{Step 3: Integral consequence.}
	Integrating over any interval \([t_1,t_2]\) contained in the regime under
	consideration yields
	\[
	\mathcal{E}_N(t_2)-\mathcal{E}_N(t_1)
	\le
	-\eta\int_{t_1}^{t_2} v(s)^2\,ds,
	\]
	which shows monotonic decay of \(\mathcal{E}_N\) and square-integrability
	of the actuated channel on that interval.
\end{proof}

\begin{remark}[Control interpretation]
	Proposition~\ref{prop:active_damping} shows that a negative correlation
	between the actuation and the velocity channel \(v=b^\top\dot q\)
	induces energy dissipation. In this sense, the learned policy can be
	interpreted as a nonlinear active damping mechanism. This behavior is
	assessed numerically in Section~\ref{sec:numerical_results} through
	energy decay and control profiles.
\end{remark}

\subsection{Comparison between PPO and the Linear Quadratic Regulator (LQR)}
\label{sec:ppo_vs_lqr}

We compare the classical linear quadratic regulator (LQR) and the
reinforcement learning-based controller obtained via Proximal Policy
Optimization (PPO) within the reduced-order modal framework. The goal is
to clarify their structural differences and to provide a consistent
baseline for the numerical experiments.

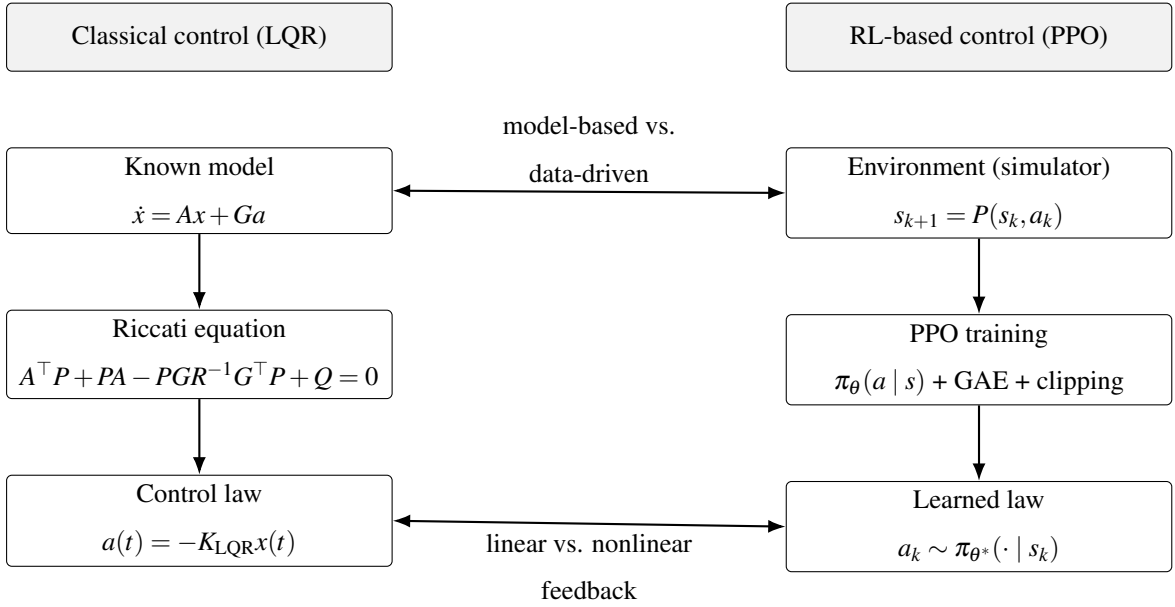
\begin{figure}[ht]
	\centering
	\begin{tikzpicture}[
		node distance=10mm and 18mm,
		>=Latex,
		box/.style={draw, rounded corners=2pt, align=center, minimum width=5.1cm, minimum height=1.1cm},
		head/.style={draw, rounded corners=2pt, align=center, minimum width=5.1cm, minimum height=0.9cm, fill=gray!10},
		arrow/.style={->, thick}
		]
		
		\node[head] (h1) {Classical control (LQR)};
		\node[box, below=of h1] (model) {Known model\\
			$\dot x=Ax+Ga$};
		\node[box, below=of model] (ric) {Riccati equation\\
			$A^\top P + P A - P G R^{-1} G^\top P + Q = 0$};
		\node[box, below=of ric] (law) {Control law\\
			$a(t) = -K_{\mathrm{LQR}}x(t)$};
		
		\draw[arrow] (model) -- (ric);
		\draw[arrow] (ric) -- (law);
		
		\node[head, right=5.2cm of h1] (h2) {RL-based control (PPO)};
		\node[box, below=of h2] (env2) {Environment (simulator)\\
			$s_{k+1}=P(s_k,a_k)$};
		\node[box, below=of env2] (ppo2) {PPO training\\
			$\pi_\theta(a\mid s)$ + GAE + clipping};
		\node[box, below=of ppo2] (law2) {Learned law\\
			$a_k\sim\pi_{\theta^\ast}(\cdot\mid s_k)$};
		
		\draw[arrow] (env2) -- (ppo2);
		\draw[arrow] (ppo2) -- (law2);
		
		\draw[<->, thick] (model.east) -- node[above, align=center] {model-based vs.\\data-driven} (env2.west);
		\draw[<->, thick] (law.east) -- node[below, align=center] {linear vs. nonlinear\\feedback} (law2.west);
		
	\end{tikzpicture}
	\caption{Conceptual comparison between LQR and PPO.}
	\label{fig:lqr_vs_ppo}
\end{figure}

\medskip
\noindent\textbf{Structural viewpoint.}
LQR is a model-based method: given the linear system \(\dot x=Ax+Ga\),
it computes an optimal linear feedback
\[
a(t)=-K_{\mathrm{LQR}}x(t),
\]
via the algebraic Riccati equation. In contrast, PPO is a data-driven
approach that learns a policy
\[
a_k\sim\pi_\theta(\cdot\mid s_k),
\]
through interaction with the simulator, without requiring an explicit
analytical model.

\medskip
\noindent\textbf{Energetic viewpoint.}
LQR minimizes a quadratic cost that directly penalizes modal energy and
control effort, yielding a Lyapunov function and exponential stability
for the unconstrained linear system. PPO, by contrast, optimizes a
finite-horizon discounted return with a reward that penalizes energy and
actuation. Although no explicit Lyapunov function is constructed, the
reward structure promotes trajectories with decreasing energy, as
discussed above.

\medskip
\noindent\textbf{Practical viewpoint (constraints and implementation).}
LQR solves an infinite-horizon optimal control problem under the
assumption of unconstrained inputs. In the present setting, the feedback
is evaluated under saturation, so that the implemented controller is no
longer strictly optimal. PPO, in contrast, is trained directly under the
same actuator constraints and naturally incorporates them into the
learned policy. Consequently, the comparison should be interpreted as a
study of two feasible controllers under identical saturation limits,
rather than a comparison of optimal solutions.

\subsection{Empirical stability criteria}

To assess closed-loop stability and performance, we monitor:
\begin{itemize}
	\item the temporal decay of \(\mathcal{E}_N(t)\);
	\item the magnitude and smoothness of the control \(a(t)\);
	\item the decay time and accumulated energy relative to the LQR baseline.
\end{itemize}

In particular, we define
\[
\mathcal{J}_E = \int_0^T \mathcal{E}_N(t)\,dt,
\qquad
\mathcal{J}_a = \int_0^T a(t)^2\,dt,
\]
and compare \((\mathcal{J}_E,\mathcal{J}_a)\) between PPO and LQR under the
same initial conditions and the same amplitude constraints.

In the numerical post-processing, these metrics are computed on a common
cropped interval \([t_0,t_{\mathrm{end}}]\) to eliminate boundary artifacts
introduced by finite-horizon effects and visual filtering. This ensures
that all controllers are compared on the same evaluation window.

\section{Numerical results}
\label{sec:numerical_results}

In this section, we report numerical experiments for the reduced-order beam
model in closed loop with a reinforcement-learning controller trained via PPO.
A deterministic LQR feedback is used as a baseline. All simulations follow the
energetic and algorithmic framework introduced in the previous sections.

A key methodological point is that we deliberately \emph{separate} (i) the
post-processing used to improve the visual readability of oscillatory time
signals in the figures from (ii) the computation of performance metrics.
Throughout, the actuator is set to the \emph{DC-capable} configuration
(\texttt{remove\_dc=False}), which is the definitive setting used in all
reported tables.

\subsection{Numerical setup}

The continuous beam model is approximated by a reduced-order model obtained via
Galerkin projection onto the bending eigenmodes. We consider odd modes only
($n=1,3,5,\dots$) and truncations $N\in\{1,2,3\}$, where $N$ denotes the number
of retained odd modes. The reduced dynamics read
\[
\ddot q(t) + \Lambda q(t) = b\,a(t), \qquad q(t)\in\mathbb{R}^{N},
\]
where $\Lambda=\mathrm{diag}(\lambda_1,\dots,\lambda_N)$ and $b\in\mathbb{R}^N$
is the rank-one boundary actuation vector induced by the chosen boundary
input channel. For $N=1$ this reduces to
\[
\ddot q_1(t) + \lambda_1 q_1(t) = b_1 a(t).
\]

The \emph{physical} modal energy is defined by
\[
\mathcal{E}_N(t)=\tfrac12\Big(\|\dot q(t)\|^2 + q(t)^\top \Lambda q(t)\Big),
\qquad
\mathcal{E}_1(t)=\tfrac12\big(\dot q_1(t)^2+\lambda_1 q_1(t)^2\big).
\]
Initial conditions are chosen as small random perturbations
$q(0),\dot q(0)\sim\mathcal{O}(10^{-2})$, yielding a nontrivial transient regime.

\noindent\textbf{Physical and discretization parameters.}
The horizon is $T=20$\,s, the time step is $\Delta t=0.01$, and the dynamics are
integrated by a fixed-step fourth-order Runge--Kutta (RK4) scheme. The control
action is saturated as $a(t)\in[-a_{\max},a_{\max}]$ with $a_{\max}=20$.

\subsection{Simulation environment and control constraints}

The environment is deterministic and fully observable. For each truncation
$N$, the state is
\[
s_k=(q_k,\dot q_k)\in\mathbb{R}^{2N},
\qquad t_k = k\Delta t, \qquad k=0,\dots,K-1,\quad K=T/\Delta t,
\]
and the scalar action approximates the boundary acceleration input,
\[
a_k \approx \ddot u(t_k), \qquad a_k\in[-a_{\max},a_{\max}].
\]
This identical saturation constraint is enforced for all controllers (PPO, LQR,
and the uncontrolled case $a\equiv 0$), ensuring a consistent comparison. In the simulator, an optional ``drift-killer'' can be activated to remove
the DC component of the control signal, effectively producing an
AC-coupled actuation model. In the experiments reported here this option
is disabled (\texttt{remove\_dc=False}), so the controllers operate in the
DC-capable setting and may exploit low-frequency components of the
actuation. The reported metrics therefore correspond to the true
saturated control signal applied to the reduced-order model.

\subsection{Controllers and evaluation protocol}

\noindent\textbf{PPO training configuration.}
PPO is implemented with \texttt{stable-baselines3}. The hyperparameters are:
\begin{itemize}
	\item rollout length $n_{\text{steps}}=2048$;
	\item minibatch size $256$;
	\item learning rate $3\times10^{-4}$;
	\item clipping parameter $\varepsilon=0.2$;
	\item discount factor $\gamma=0.99$;
	\item GAE parameter $\lambda=0.95$;
	\item gradient clipping $\texttt{max\_grad\_norm}=0.5$;
	\item entropy coefficient $\texttt{ent\_coef}=0$.
\end{itemize}
Observations are normalized online using \texttt{VecNormalize}. This affects
only the network input and does \emph{not} modify the physical variables
$(q,\dot q)$ used to compute energies and metrics.

The policy and value function are parameterized by a two-layer MLP with
architecture $64$--$64$ and $\tanh$ activations. The reward penalizes
instantaneous energy and control effort, namely
\[
r_k = -\Big(c_E\,\mathcal{E}_N(t_k) + \mu\,a_k^2\Big),
\]
with fixed weights $c_E=1$ and $\mu=10^{-2}$.

\noindent\textbf{LQR baseline.}
For each truncation $N\in\{1,2,3\}$, an LQR controller is designed for the
corresponding $2N$-dimensional state $x=(q,\dot q)$ using a quadratic cost that
penalizes modal displacement, modal velocity, and control effort. The resulting
linear feedback $a(t)=-Kx(t)$ is then implemented with the \emph{same}
saturation $|a|\le a_{\max}$ as PPO. Hence the LQR baseline is optimal for the
associated \emph{unsaturated} quadratic problem, but is evaluated here only
after projection onto the same hard actuation bounds imposed on PPO.

\medskip
\noindent\textbf{Common evaluation window.}
To avoid initialization and end-of-horizon artifacts in time-series summaries,
all metrics are computed on the common cropped interval
\[
[t_0,t_{\mathrm{end}}],\qquad t_0=3\tau_{\mathrm{hp}},\qquad
t_{\mathrm{end}}=T-\texttt{extra\_end},
\]
with the numerical values $\tau_{\mathrm{hp}}=0.25$ and $\texttt{extra\_end}=1.0$,
so that $t_0=0.75$ and $t_{\mathrm{end}}=19.0$.

For visualization purposes, the time-series shown in
Figures~\ref{fig:q1_ppo}--\ref{fig:control_ppo} are displayed after
high-pass filtering and using the same crop window. This filtering
removes quasi-static or low-frequency trends and improves the
readability of the oscillatory dynamics. However, all reported metrics
and tables are computed from the unfiltered physical trajectories
$(q(t),\dot q(t),a(t))$ and from the physical energy $\mathcal{E}_N(t)$.
No filtered signals are used in the computation of the performance
indices such as $\mathcal{J}_E$.
\subsection{Performance metrics}

Closed-loop performance is quantified on $[t_0,t_{\mathrm{end}}]$ by:
\begin{itemize}
	\item physical modal energy $\mathcal{E}_N(t)$;
	\item accumulated energy
	\[
	\mathcal{J}_E = \int_{t_0}^{t_{\mathrm{end}}} \mathcal{E}_N(t)\,dt;
	\]
	\item accumulated control effort
	\[
	\mathcal{J}_a = \int_{t_0}^{t_{\mathrm{end}}} a(t)^2\,dt;
	\]
	\item peak control amplitude $\|a\|_\infty$ on the same interval;
	\item decay times $t_{10\%}$ and $t_{1\%}$ defined as the first times in
	$[t_0,t_{\mathrm{end}}]$ such that $\mathcal{E}_N(t)\le 0.1\,\mathcal{E}_N(t_0)$
	and $\mathcal{E}_N(t)\le 0.01\,\mathcal{E}_N(t_0)$, respectively.
\end{itemize}

The numerical values are reported in Table~\ref{tab:metrics_ppo_lqr} for the
single-mode comparison $N=1$, and in Table~\ref{tab:metrics_multimode} for the
full comparison across truncations $N=1,2,3$, including uncontrolled dynamics,
saturated LQR, and PPO.

\subsection{Results and discussion: single-mode case ($N=1$)}

\noindent\textbf{Closed-loop mechanism and modal vibration control.}
The filtered time histories in Figures~\ref{fig:q1_ppo}--\ref{fig:control_ppo}
provide a physically interpretable picture of the learned feedback in the
single-mode regime. The modal coordinate $q_1(t)$ exhibits a progressive decay
of oscillation amplitude, the energy-like filtered signal shows attenuation of
the oscillatory content, and the control action remains bounded while oscillating
with a phase relation consistent with energy extraction from the mode.

This behavior is characteristic of \emph{active modal vibration control}.
Indeed, for $N=1$ the reduced dynamics
\[
\ddot q_1(t)+\lambda_1 q_1(t)=b_1 a(t)
\]
describe a single controlled oscillator, and the PPO policy effectively learns a
state-dependent forcing law that, on average, removes energy from the modal
motion. In this sense, the learned feedback acts as a nonlinear damping
mechanism rather than merely as a generic black-box controller.

Although Figures~\ref{fig:q1_ppo}--\ref{fig:control_ppo} are shown only for
$N=1$, the same interpretation remains qualitatively relevant for higher
truncations. For $N=2$, the reduction of $\mathcal{J}_E$ and the corresponding
energy plots indicate that PPO still performs net vibration attenuation of the
dominant modal content, albeit less effectively than in the single-mode case.
For $N=3$, by contrast, the metrics and energy curves show that this damping
mechanism is no longer sufficiently coordinated across all retained modes,
so that the clear attenuation effect observed for $N=1$ is lost.

The high-pass representation is used here purely to suppress slow components and
make the phase relation visually readable; the quantitative conclusions,
however, are those reported in Table~\ref{tab:metrics_ppo_lqr}, which are
computed from the unfiltered physical signals.

\begin{figure}[h!]
	\centering
	\includegraphics[width=0.92\textwidth]{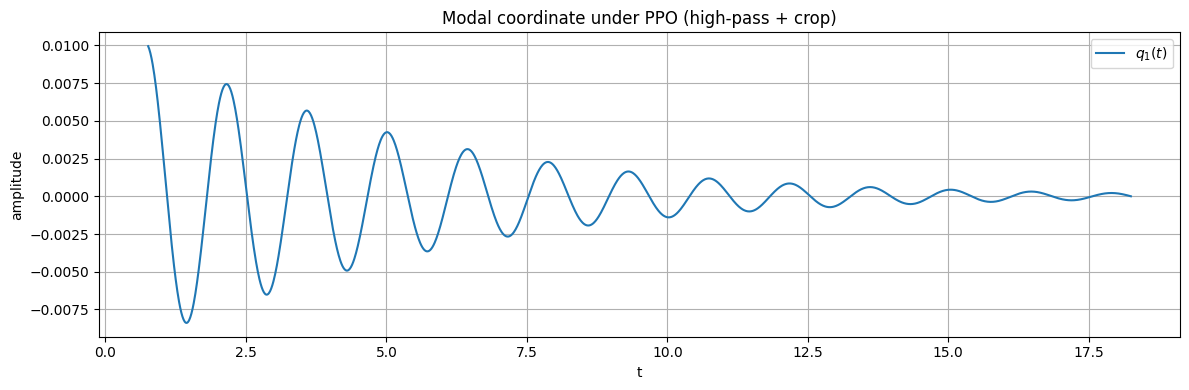}
	\caption{PPO closed-loop response (high-pass + crop, for visualization): modal coordinate $q_1(t)$.}
	\label{fig:q1_ppo}
\end{figure}

\begin{figure}[h!]
	\centering
	\includegraphics[width=0.92\textwidth]{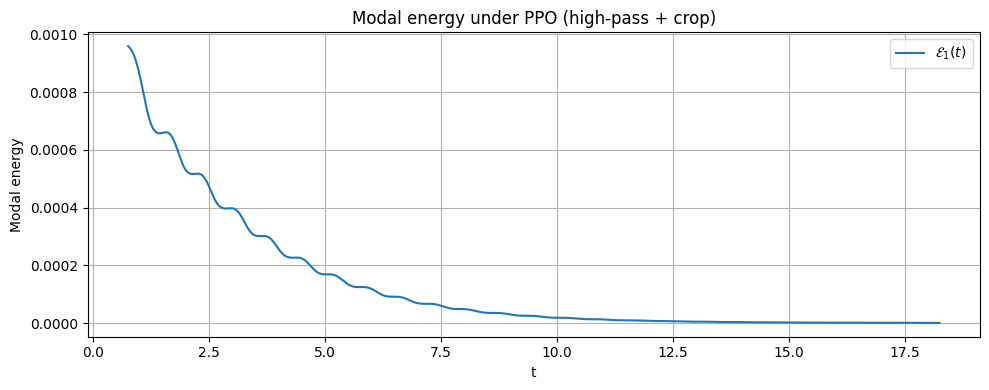}
	\caption{PPO closed-loop response (high-pass + crop, for visualization): modal energy-like signal highlighting oscillatory decay.}
	\label{fig:energy_ppo}
\end{figure}

\begin{figure}[h!]
	\centering
	\includegraphics[width=0.92\textwidth]{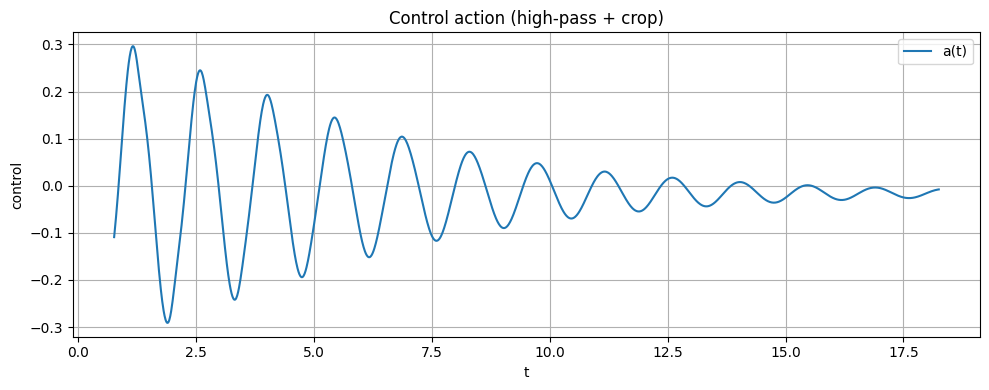}
	\caption{PPO closed-loop response (high-pass + crop, for visualization): control action $a(t)$.}
	\label{fig:control_ppo}
\end{figure}

Figure~\ref{fig:energy_compare} and Table~\ref{tab:metrics_ppo_lqr} show that,
in the single-mode regime, PPO clearly outperforms both the uncontrolled
response and the saturated LQR baseline, while respecting the same hard
actuation bounds.

\noindent\textbf{Energy attenuation and transient times.}
Compared with the uncontrolled dynamics, PPO reduces the accumulated energy from
$\mathcal{J}_E=6.719\times10^{-2}$ to $\mathcal{J}_E=9.005\times10^{-3}$,
corresponding to an energy reduction of approximately $86.6\%$ over the common
evaluation window. Relative to the saturated LQR baseline, PPO further reduces
$\mathcal{J}_E$ by about $61.7\%$
($2.354\times10^{-2}\mapsto9.005\times10^{-3}$). This superiority is also
visible in the transient decay indicators: PPO reaches the $10\%$ level at
$t_{10\%}\approx 7.58$\,s, whereas saturated LQR only reaches this threshold at
$t_{10\%}\approx 18.58$\,s, i.e.\ essentially at the end of the horizon. In
addition, PPO reaches the stricter $1\%$ level within the horizon
($t_{1\%}\approx 13.97$\,s), while LQR does not reach this threshold at all
under saturation.

\noindent\textbf{Control-effort trade-off under saturation.}
This improved damping performance is obtained at the price of larger actuation.
Indeed, PPO uses more integrated control effort than LQR,
$\mathcal{J}_a=1.256\times10^{-1}$ versus $4.700\times10^{-2}$, and also
exhibits a larger peak amplitude,
$\|a\|_\infty=2.474\times10^{-1}$ versus $1.132\times10^{-1}$. Hence, in the
single-mode setting, PPO acts as a more aggressive but substantially more
effective nonlinear damping mechanism than the saturated linear-quadratic
baseline.

\begin{figure}[h!]
	\centering
	\includegraphics[width=0.92\textwidth]{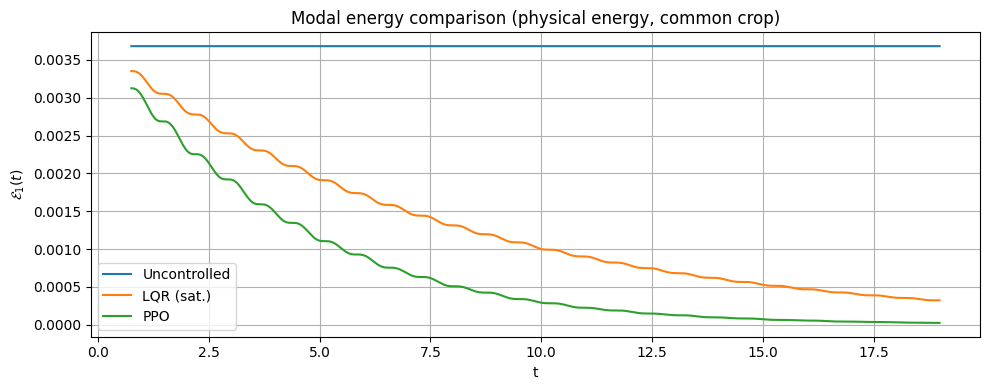}
	\caption{Modal energy comparison (physical energy, common crop): uncontrolled, saturated LQR, and PPO for $N=1$.}
	\label{fig:energy_compare}
\end{figure}

\begin{table}[h!]
	\centering
	\caption{Performance metrics for $N=1$ on the common cropped interval $[t_0,t_{\mathrm{end}}]$ (physical energy, \texttt{remove\_dc=False}).}
	\label{tab:metrics_ppo_lqr}
	\begin{tabular}{lccccc}
		\toprule
		Controller
		& $\mathcal{J}_E$
		& $\mathcal{J}_a$
		& $\|a\|_\infty$
		& $t_{10\%}$ (s)
		& $t_{1\%}$ (s)\\
		\midrule
		Uncontrolled & 6.719e-02 & 0.000e+00 & 0.000e+00 & --        & --        \\
		LQR (sat.)   & 2.354e-02 & 4.700e-02 & 1.132e-01 & 1.858e+01 & --        \\
		PPO          & 9.005e-03 & 1.256e-01 & 2.474e-01 & 7.580e+00 & 1.397e+01 \\
		\bottomrule
	\end{tabular}
\end{table}

\subsection{Multimodal truncations and performance degradation ($N=2,3$)}

We now assess how the same PPO architecture behaves as the truncation dimension
increases. Table~\ref{tab:metrics_multimode} and
Figure~\ref{fig:energy_multimode} indicate a progressive loss of control quality
from the single-mode case to higher-dimensional truncations.

\noindent\textbf{$N=2$: quantitative gain in accumulated energy, but no strong decay regime.}
For $N=2$, PPO still improves upon both baselines in terms of accumulated
energy. In particular, $\mathcal{J}_E$ decreases from
$1.566\times10^{-1}$ for the uncontrolled dynamics to
$8.873\times10^{-2}$ for PPO, which corresponds to an improvement of about
$43.3\%$. PPO also improves over saturated LQR, reducing $\mathcal{J}_E$ from
$1.050\times10^{-1}$ to $8.873\times10^{-2}$, i.e.\ by roughly $15.5\%$.
Thus, for $N=2$, PPO yields the smallest accumulated energy among the three
controllers.

However, neither PPO nor LQR reaches the $10\%$ or $1\%$ decay thresholds within
the evaluation horizon. This shows that, although PPO still achieves net modal
vibration attenuation in an integral sense, the strong damping regime observed
for $N=1$ is no longer present. The improved energy metric is again accompanied
by higher effort: PPO requires
$\mathcal{J}_a=2.559\times10^{-1}$ and $\|a\|_\infty=3.568\times10^{-1}$,
both larger than the corresponding LQR values
$\mathcal{J}_a=2.124\times10^{-1}$ and $\|a\|_\infty=2.206\times10^{-1}$.

\noindent\textbf{$N=3$: loss of PPO superiority and near-neutral energy performance.}
For $N=3$, the PPO policy no longer provides a meaningful energy advantage.
Indeed, the accumulated energy is slightly larger than in the uncontrolled case,
$\mathcal{J}_E=3.049$ versus $3.004$, so PPO does not achieve net damping
improvement on this truncation. On the other hand, PPO is considerably less
aggressive than the saturated LQR baseline: while LQR attains a slightly smaller
energy value ($\mathcal{J}_E=2.928$), it does so at the expense of a much larger
control effort, with $\mathcal{J}_a=5.848$ and
$\|a\|_\infty=9.442\times10^{-1}$, compared to
$\mathcal{J}_a=8.553\times10^{-1}$ and
$\|a\|_\infty=3.962\times10^{-1}$ for PPO. Therefore, the $N=3$ case is better
interpreted not as a clear instability mechanism, but as a regime in which the
learned policy is no longer sufficiently effective to coordinate all retained
modal vibrations and preserve the strong damping advantage seen for lower
truncations.

\noindent\textbf{Why multimodal control becomes harder.}
This progressive deterioration is consistent with three structural features of
the problem:
\begin{itemize}
	\item \textbf{Rank-one boundary actuation:} the single control input $a(t)$ must
	act simultaneously on all retained modal coordinates through the fixed vector
	$b$, which severely limits the number of independently damped directions as
	$N$ increases.
	\item \textbf{Separated modal time scales:} higher truncations involve faster
	oscillatory components, so that a single stationary policy must respond across
	distinct frequencies and phase relations.
	\item \textbf{Fixed policy capacity and training budget:} using the same
	$64$--$64$ architecture and the same learning budget for $N=1,2,3$ does not
	scale with the increased state-space complexity.
\end{itemize}
These mechanisms explain why PPO is highly effective for $N=1$, still beneficial
for $N=2$, but loses its clear advantage for $N=3$ under the present
configuration.

\begin{table}[h!]
	\centering
	\caption{Performance metrics across truncation orders on the common cropped interval $[t_0,t_{\mathrm{end}}]$ (physical energy, \texttt{remove\_dc=False}).}
	\label{tab:metrics_multimode}
	\begin{tabular}{ccccccc}
		\toprule
		$N$ & Controller & $\mathcal{J}_E$ & $\mathcal{J}_a$ & $\|a\|_\infty$ & $t_{10\%}$ (s) & $t_{1\%}$ (s)\\
		\midrule
		1 & Uncontrolled & 6.719e-02 & 0.000e+00 & 0.000e+00 & -- & -- \\
		1 & LQR (sat.)   & 2.354e-02 & 4.700e-02 & 1.132e-01 & 1.858e+01 & -- \\
		1 & PPO          & 9.005e-03 & 1.256e-01 & 2.474e-01 & 7.580e+00 & 1.397e+01 \\
		\midrule
		2 & Uncontrolled & 1.566e-01 & 0.000e+00 & 0.000e+00 & -- & -- \\
		2 & LQR (sat.)   & 1.050e-01 & 2.124e-01 & 2.206e-01 & -- & -- \\
		2 & PPO          & 8.873e-02 & 2.559e-01 & 3.568e-01 & -- & -- \\
		\midrule
		3 & Uncontrolled & 3.004e+00 & 0.000e+00 & 0.000e+00 & -- & -- \\
		3 & LQR (sat.)   & 2.928e+00 & 5.848e+00 & 9.442e-01 & -- & -- \\
		3 & PPO          & 3.049e+00 & 8.553e-01 & 3.962e-01 & -- & -- \\
		\bottomrule
	\end{tabular}
\end{table}

\begin{figure}[h!]
	\centering
	\includegraphics[width=0.92\textwidth]{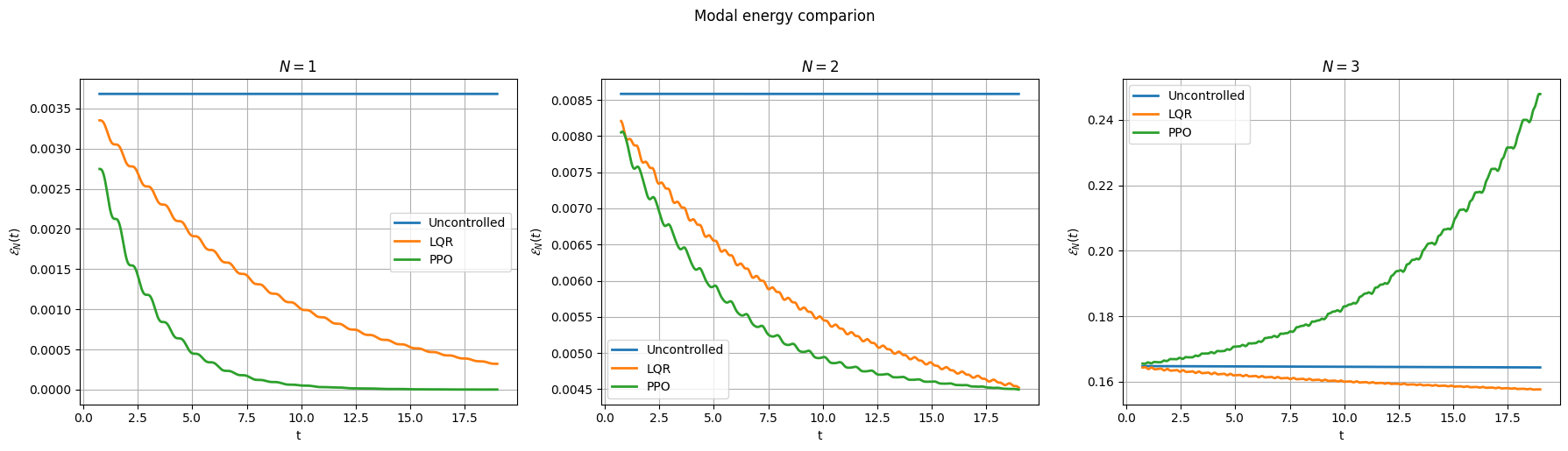}
	\caption{Modal energy comparison for $N=1,2,3$: in each truncation, the three curves correspond to uncontrolled dynamics, saturated LQR, and PPO on the common cropped interval.}
	\label{fig:energy_multimode}
\end{figure}

\noindent\textbf{Run-to-run variability.}
Even with fixed seeds, PPO training may exhibit moderate variability due to
the nonconvex neural optimization process and the use of observation
normalization. The qualitative ranking observed here strong improvement
for $N=1$, weaker but still positive performance for $N=2$, and loss of
superiority for $N=3$  appears robust across runs, although publication-level
reporting should include averages and standard deviations over multiple
independent trainings.

\noindent\textbf{Scope.}
The present study focuses on low-order reduced-order models, where the
dominant beam dynamics are typically concentrated in the first modes.
The deterioration observed for $N\geq3$ does not invalidate the successful
single-mode result; rather, it highlights the increased difficulty of
multimodal coordination under rank-one boundary actuation. Addressing this
regime may require richer policy architectures, larger training budgets,
or alternative actuation mechanisms.


\section{Conclusion}

In this study, boundary control via reinforcement learning was investigated for a modified Timoshenko system, which was approximated through a Galerkin truncation. Initially, results concerning well-posedness, existence, and uniqueness were presented for the considered problem. Subsequently, the control problem was formulated in terms of a Markov Decision Process, and an analytical study of the PPO algorithm and its characterization within the context of the considered control problem was carried out.

Next, the closed-loop energy stability was investigated in order to interpret the role of PPO in the energy stabilization process, including a comparison with an optimal baseline given by the LQR controller. In the numerical part, experiments were conducted that support the theoretical analysis developed in the previous sections, from which the following conclusions were obtained.

For $N=1$, PPO substantially outperforms both the uncontrolled response and the saturated LQR baseline in terms of physical energy attenuation, reducing $\mathcal{J}_E$ and reaching both $t_{10\%}$ and $t_{1\%}$ well within the evaluation horizon, although at the cost of higher control effort. In this regime, the learned policy clearly behaves as an effective modal vibration controller with a nonlinear damping character.

For $N=2$, PPO remains beneficial and improves upon both the uncontrolled dynamics and the saturated LQR in terms of accumulated energy, but it no longer induces a strong decay regime capable of reaching the $10\%$ or $1\%$ thresholds. Nevertheless, the controller still attenuates the dominant modal vibrations, although less sharply than in the single-mode case.

For $N=3$, PPO loses its energetic advantage: its accumulated energy becomes slightly worse than in the uncontrolled case, although it requires significantly less effort than the saturated LQR. This indicates that, under the architecture and training budget adopted here, the learned feedback is no longer sufficiently coordinated to control higher-dimensional modal interactions with the same effectiveness observed in lower truncations.

Overall, the results show that boundary control via reinforcement learning for the reduced-order model (ROM) of the Timoshenko system is highly effective for the initial modes, but exhibits a loss of performance for higher-frequency modes. This, however, does not diminish the relevance of the results, since the first modes often dominate the main oscillatory characteristics of beam vibrations.

Finally, the obtained results indicate that reinforcement learning, particularly through PPO, constitutes a promising approach for vibration control problems in distributed systems, especially when the objective is centered on the dominant modal dynamics. At the same time, the results point out important limitations in more complex multimodal regimes, opening several directions for future research. From a theoretical perspective, further investigations may focus on establishing stability guarantees for learned feedback policies and analyzing the relationship between the reduced-order dynamics and the original PDE system. From a numerical perspective, future work may explore more expressive policy architectures, refined reward strategies, different actuation mechanisms, and more extensive computational experiments to assess the robustness and scalability of the proposed approach.

\newpage
{\bf Acknowledgments:}

We would like to thank the anonymous reviewers for their careful reading and valuable suggestions, which helped to improve the clarity and overall presentation of the paper. We are grateful for the constructive feedback received, which significantly contributed to important improvements in the manuscript.

The author also acknowledges the Instituto de Matemática Pura e Aplicada (IMPA) for providing an inspiring and stimulating research environment during the summer postdoctoral visit, which greatly contributed to the development of this work.
\vspace{0.5cm}

{\bf Funding:}


\end{document}